\documentclass[11pt,twoside,a4paper]{article}
\usepackage{makeidx}
\usepackage[dvips]{graphicx}
\usepackage{amscd}
\usepackage{amsmath}
\usepackage{amssymb}
\usepackage{latexsym}
\usepackage{stmaryrd}
\usepackage[latin1]{inputenc}
\usepackage[T1]{fontenc}   
\usepackage[english]{babel}
\newcommand{\tun}{\begin{picture}(5,0)(-2,-1)
\put(0,0){\circle*{2}}
\end{picture}}

\newcommand{\tdeux}{\begin{picture}(7,7)(0,-1)
\put(3,0){\circle*{2}}
\put(3,0){\line(0,1){5}}
\put(3,5){\circle*{2}}
\end{picture}}

\newcommand{\ttroisun}{\begin{picture}(15,12)(-5,-1)
\put(3,0){\circle*{2}}
\put(-0.65,0){$\vee$}
\put(6,7){\circle*{2}}
\put(0,7){\circle*{2}}
\end{picture}}
\newcommand{\ttroisdeux}{\begin{picture}(5,15)(-2,-1)
\put(0,0){\circle*{2}}
\put(0,0){\line(0,1){5}}
\put(0,5){\circle*{2}}
\put(0,5){\line(0,1){5}}
\put(0,10){\circle*{2}}
\end{picture}}

\newcommand{\tquatreun}{\begin{picture}(15,12)(-5,-1)
\put(3,0){\circle*{2}}
\put(-0.65,0){$\vee$}
\put(6,7){\circle*{2}}
\put(0,7){\circle*{2}}
\put(3,7){\circle*{2}}
\put(3,0){\line(0,1){7}}
\end{picture}}
\newcommand{\tquatredeux}{\begin{picture}(15,18)(-5,-1)
\put(3,0){\circle*{2}}
\put(-0.65,0){$\vee$}
\put(6,7){\circle*{2}}
\put(0,7){\circle*{2}}
\put(0,14){\circle*{2}}
\put(0,7){\line(0,1){7}}
\end{picture}}
\newcommand{\tquatretrois}{\begin{picture}(15,18)(-5,-1)
\put(3,0){\circle*{2}}
\put(-0.65,0){$\vee$}
\put(6,7){\circle*{2}}
\put(0,7){\circle*{2}}
\put(6,14){\circle*{2}}
\put(6,7){\line(0,1){7}}
\end{picture}}
\newcommand{\tquatrequatre}{\begin{picture}(15,18)(-5,-1)
\put(3,5){\circle*{2}}
\put(-0.65,5){$\vee$}
\put(6,12){\circle*{2}}
\put(0,12){\circle*{2}}
\put(3,0){\circle*{2}}
\put(3,0){\line(0,1){5}}
\end{picture}}
\newcommand{\tquatrecinq}{\begin{picture}(9,19)(-2,-1)
\put(0,0){\circle*{2}}
\put(0,0){\line(0,1){5}}
\put(0,5){\circle*{2}}
\put(0,5){\line(0,1){5}}
\put(0,10){\circle*{2}}
\put(0,10){\line(0,1){5}}
\put(0,15){\circle*{2}}
\end{picture}}

\newcommand{\ptroisun}{\begin{picture}(15,12)(-5,-1)
\put(3,7){\circle*{2}}
\put(-0.65,0){$\wedge$}
\put(6,0){\circle*{2}}
\put(0,0){\circle*{2}}
\end{picture}}

\newcommand{\pquatreun}{\begin{picture}(15,12)(-5,-1)
\put(3,7){\circle*{2}}
\put(-0.65,0){$\wedge$}
\put(6,0){\circle*{2}}
\put(0,0){\circle*{2}}
\put(3,0){\circle*{2}}
\put(2.9,0){\line(0,1){7}}
\end{picture}}
\newcommand{\pquatredeux}{\begin{picture}(15,18)(-5,-1)
\put(3,14){\circle*{2}}
\put(-0.65,7){$\wedge$}
\put(6,7){\circle*{2}}
\put(0,7){\circle*{2}}
\put(0,0){\circle*{2}}
\put(0,0){\line(0,1){7}}
\end{picture}}
\newcommand{\pquatretrois}{\begin{picture}(15,18)(-5,-1)
\put(3,14){\circle*{2}}
\put(-0.65,7){$\wedge$}
\put(6,7){\circle*{2}}
\put(0,7){\circle*{2}}
\put(6,0){\circle*{2}}
\put(6,0){\line(0,1){7}}
\end{picture}}
\newcommand{\pquatrequatre}{\begin{picture}(15,18)(-5,-1)
\put(3,7){\circle*{2}}
\put(-0.65,0){$\wedge$}
\put(6,0){\circle*{2}}
\put(0,0){\circle*{2}}
\put(3,12){\circle*{2}}
\put(3,7){\line(0,1){5}}
\end{picture}}
\newcommand{\pquatrecinq}{\begin{picture}(15,12)(-5,-1)
\put(0,0){\circle*{2}}
\put(7,0){\circle*{2}}
\put(0,7){\circle*{2}}
\put(7,7){\circle*{2}}
\put(0,0){\line(0,1){7}}
\put(7,0){\line(0,1){7}}
\put(.5,1.5){$\scriptstyle \diagup$}
\end{picture}}
\newcommand{\pquatresix}{\begin{picture}(15,12)(-5,-1)
\put(0,0){\circle*{2}}
\put(7,0){\circle*{2}}
\put(0,7){\circle*{2}}
\put(7,7){\circle*{2}}
\put(0,0){\line(0,1){7}}
\put(7,0){\line(0,1){7}}
\put(0,1.5){$\scriptstyle \diagdown$}
\end{picture}}
\newcommand{\pquatresept}{\begin{picture}(15,12)(-5,-1)
\put(0,0){\circle*{2}}
\put(7,0){\circle*{2}}
\put(0,7){\circle*{2}}
\put(7,7){\circle*{2}}
\put(0,0){\line(0,1){7}}
\put(7,0){\line(0,1){7}}
\put(.5,1.5){$\scriptstyle \diagup$}
\put(0,1.5){$\scriptstyle \diagdown$}
\end{picture}}
\newcommand{\pquatrehuit}{\begin{picture}(15,18)(-5,-1)
\put(3,0){\circle*{2}}
\put(-0.65,0){$\vee$}
\put(6,7){\circle*{2}}
\put(0,7){\circle*{2}}
\put(3,14){\circle*{2}}
\put(-0.65,7){$\wedge$}
\end{picture}}

\input{xy}
\xyoption{all}

\newcommand{\FQSym}{\mathbf{FQSym}}
\newcommand{\WNP}{\mathcal{WNP}}
\newcommand{\PP}{\mathcal{PP}}
\newcommand{\PF}{\mathcal{PF}}
\renewcommand{\P}{\mathcal{P}}
\newcommand{\h}{{\cal H}}
\renewcommand{\S}{\mathfrak{S}}
\newcommand{\prodg}{\rightsquigarrow}
\newcommand{\prodh}{\lightning}
\newcommand{\tdelta}{\tilde{\Delta}}
\newcommand{\addots}{\begin{picture}(8,8)(0,0)\put(0,0){.}\put(4,4){.}\put(8,8){.}\end{picture}}

\title{Deformation of the Hopf algebra of plane posets}
\date{}
\author{Lo{\"\i}c Foissy \\
\\
{\small{\it Laboratoire de Mathématiques, Université de Reims}}\\
\small{{\it Moulin de la Housse - BP 1039 - 51687 REIMS Cedex 2, France}}\\
\small{e-mail : loic.foissy@univ-reims.fr}}

\newtheorem{defi}{\indent Definition}
\newtheorem{lemma}[defi]{\indent Lemma}
\newtheorem{cor}[defi]{\indent Corollary}
\newtheorem{theo}[defi]{\indent Theorem}
\newtheorem{prop}[defi]{\indent Proposition}

\newenvironment{proof}{{\bf Proof.}}{\hfill $\Box$}

\begin{document}

\maketitle

ABSTRACT. We describe and study a four parameters deformation of the two products and the coproduct of the Hopf algebra of plane posets.
We obtain a family of braided Hopf algebras, which are generically self-dual. We also prove that in a particular case (when the second parameter goes to zero
and the first and third parameters are equal), this deformation is isomorphic, as a self-dual braided Hopf algebra, to a deformation of the Hopf algebra
of free quasi-symmetric functions $\FQSym$. \\
 

AMS CLASSIFICATION. 06A11, 16W30, 16S80.

\tableofcontents

\section*{Introduction}

 A double poset is a finite set with two partial orders. As explained in \cite{MR2}, the space generated by the double posets
 inherits two products and one coproduct, here denoted by $m$, $\prodh$ and $\Delta$, 
making it both a Hopf and an infinitesimal Hopf algebra \cite{Loday}.  Moreover, this Hopf algebra is self dual. A double poset is plane if its two partial orders satisfy 
a (in)compatibility condition, see definition \ref{1}.  The subspace $\h_{\PP}$ generated by plane posets is stable under the two products and the coproduct,
and is self-dual \cite{F1,F2}: in particular, two Hopf pairings are defined on it, using the notion of picture.
Moreover, as proved in \cite{F2}, it is isomorphic to the Hopf algebra of free quasi-symmetric functions $\FQSym$, also known as
the Malvenuto-Reutenauer Hopf algebra of permutations. An explicit isomorphism $\Theta$ is given by the linear extensions of plane posets, see definition \ref{11}.\\

We define in this text a four-parameters deformation of the products and the coproduct of $\h_{\PP}$, together with a deformation of the two pairings
and of the morphism from $\h_{\PP}$ to $\FQSym$. If $q=(q_1,q_2,q_3,q_4)\in K^4$, the product $m_q(P \otimes Q)$ of two plane posets $P$ and $Q$ 
is a linear span of plane posets $R$ such that $R=P\sqcup Q$ as a set, $P$ and $Q$ being plane subposets of $R$. The coefficients are defined with the help of
the two partial orders of $R$, see theorem \ref{14}, and are polynomials in $q$. In particular:
$$\left\{\begin{array}{rcl}
m_{(1,0,0,0)} &=&\prodh,\\
m_{(0,1,0,0)}&=&\prodh^{op},\\
m_{(0,0,1,0)}&=&m,\\
m_{(0,0,0,1)}&=&m^{op}.
\end{array}\right.$$
We also obtain the product dual to the coproduct $\Delta$ (considering the basis of double posets
as orthonormal) as $m_{(1,0,1,1)}$, and its opposite given by $m_{(0,1,1,1)}$. 
Dually, we define a family of coassociative coproducts $\Delta_q$. For any plane poset $P$, $\Delta(P)$ is a linear span of terms $(P\setminus I)\otimes I$,
running over the plane subposets $I$ of $P$, the coefficients being polynomials in $q$. In the particular cases where at least one of the components of $q$ is zero, 
only $h$-ideals, $r$-ideals or biideals can appear in this sum (definition \ref{7} and proposition \ref{22}). We study the compatibility of $\Delta_q$ 
with both products $m$ and $\prodh$ on $\h_{\PP}$ (proposition \ref{23}). For example, $(\h_{\PP},m,\Delta_q)$ satisfies the axiom
$$\Delta_q(xy)=\sum\sum q_3^{|x'_q||y''_q|}q_4^{|x''_q||y'_q|} (x'_qy'_q) \otimes (x''_qy''_q).$$
If $q_3=1$, it is a braided Hopf algebra, with braiding given by $c_q(P\otimes Q)=q_4^{|P||Q|} Q \otimes P$; in particular, if $q_3=1$ and $q_4=1$ 
this is a Hopf algebra, and if $q_3=1$ and $q_4=0$ it is an infinitesimal Hopf algebra. If $q_4=1$, it is the coopposite (or the opposite) 
of a braided Hopf algebra. Similar results hold if we consider the second product $\prodh$, permuting the roles of $(q_3,q_4)$ and $(q_1,q_2)$. 

We define a symmetric pairing $\langle-,-\rangle_q$ such that:
$$\langle x\otimes y,\Delta_q(z)\rangle_q=\langle xy,z\rangle_q\mbox{ for all }x,y,z\in \h_{\PP}.$$
If $q=(1,0,1,1)$, we recover the first "classical" pairing of $\h_{\PP}$. We prove that in the case $q_2=0$, 
this pairing is nondegenerate if, and only if, $q_1\neq 0$ (corollary \ref{36}). Consequently, this pairing is generically nondegenerate. 

The coproduct of $\FQSym$ is finally deformed, in such a way that  the algebra morphism $\Theta$ from $\h_{\PP}$ to $\FQSym$ becomes compatible
with $\Delta_q$, if $q$ has the form $q=(q_1,0,q_1,q_4)$. Deforming the second pairing $\langle-,-\rangle'$ of $\h_{\PP}$ and the usual Hopf pairing
of $\FQSym$, the map $\Theta$ becomes also an isometry (theorem \ref{40}). Consequently, the deformation $\langle-,-\rangle'_q$ is nondegenerate if, and only if,
$q_1q_4 \neq 0$.\\

This text is organized in the following way. The first section contains reminders on the Hopf algebra of plane posets $\h_{\PP}$, its two products, its coproducts
and its two Hopf pairings, and on the isomorphism $\Theta$ from $\h_{\PP}$ to $\FQSym$. The deformation of the products is defined in section 2,
and we also consider the compatibility of these products with several natural bijections on $\PP$ and the stability of certain families of plane posets under these products.
We proceed to the dual construction in the next section, where we also study the compatibility with the two (undeformed) products.
The deformation of the first pairing is described in section 4. The compatibilities with the bijections on $\PP$ or with the second product $\prodh$ are also given,
and the nondegeneracy is proved for $q=(q_1,0,q_3,q_4)$ if $q_1\neq 0$. The last section is devoted to the deformation of the second pairing
and of the morphism to $\FQSym$.
\section{Recalls and notations}

\label{recalls}

\subsection{Double and plane posets}

\begin{defi}\label{1}\begin{enumerate}
\item \cite{MR2} A \emph{ double poset} is a triple $(P,\leq_1,\leq_2)$, where $P$ is a finite set and $\leq_1$, $\leq_2$ are two partial orders on $P$.
\item A \emph{ plane poset} is a double poset $(P,\leq_h,\leq_r)$ such that for all $x, y\in P$, such that $x\neq y$, $x$ and $y$ are comparable for $\leq_h$ 
if, and only if, $x$ and $y$ are not comparable for $\leq_r$. The set of (isoclasses of) plane posets will be denoted by $\PP$. 
For all $n \in \mathbb{N}$, the set of (isoclasses of) plane posets of cardinality $n$ will be denoted by $\PP(n)$.
\end{enumerate}\end{defi}

{\bf Examples}. Here are the plane posets of cardinal $\leq 4$. They are given by the Hasse graph of $\leq_h$; 
if $x$ and $y$ are two vertices of this graph which are not comparable for $\leq_h$, then $x \leq_r y$ if $y$ is more on the right than $x$.
\begin{eqnarray*}
\PP(0)&=&\{1\},\\
\PP(1)&=&\{\tun\},\\
\PP(2)&=&\{\tun\tun,\tdeux\},\\
\PP(3)&=&\{\tun\tun\tun,\tun\tdeux,\tdeux\tun,\ttroisun,\ttroisdeux,\ptroisun\},\\
\PP(4)&=&\left\{\begin{array}{c}
\tun\tun\tun\tun,\tun\tun\tdeux,\tun\tdeux\tun,\tdeux\tun\tun,\tun\ttroisun,\ttroisun\tun,\tun\ttroisdeux,\ttroisdeux\tun,\tun\ptroisun,\ptroisun\tun,\tdeux\tdeux,\\
\tquatreun,\tquatredeux,\tquatretrois,\tquatrequatre,\tquatrecinq,\pquatreun,\pquatredeux,\pquatretrois,\pquatrequatre,\pquatrecinq,\pquatresix,\pquatresept,\pquatrehuit
\end{array}\right\}.\end{eqnarray*}

The following proposition is proved in \cite{F1} (proposition 11):

\begin{prop}
Let $P\in\PP$. We define a relation $\leq$ on $P$ by:
$$(x\leq y) \mbox{ if } (x\leq_h y\mbox{ or } x \leq_r y).$$
Then $\leq$ is a total order on $P$.
\end{prop}

As a consequence, substituing $\leq$ to $\leq_r$, plane posets are also \emph{special} posets \cite{MR2}, that is to say double posets such that the second
order is total. For any plane poset $P\in \PP(n)$, we shall assume that $P=\{1,\ldots,n\}$ as a totally ordered set. \\

The following theorem is proved in \cite{F2}:

\begin{theo}
Let $\sigma$ be a permutation in the $n$th symmetric group $\S_n$. We define a plane poset $P_\sigma$ in the following way:
\begin{itemize}
\item $P_\sigma=\{1,\ldots,n\}$ as a set.
\item If $i,j \in P_\sigma$, $i\leq_h j$ if $i\leq j$ and $\sigma(i) \leq \sigma(j)$.
\item If $i,j \in P_\sigma$, $i\leq_r j$ if $i\leq j$ and $\sigma(i) \geq \sigma(j)$.
\end{itemize}
Note that the total order on $\{1,\ldots,n\}$ induced by this plane poset structure is the usual one.
Then for all $n \geq 0$, the following map is a bijection:
$$\Psi_n: \left\{\begin{array}{rcl}
\S_n&\longrightarrow&\PP(n)\\
\sigma&\longrightarrow&P_\sigma.
\end{array}\right.$$ \end{theo}

{\bf Examples.}
$$\begin{array}{rclccrclccrclccrcl}
\Psi_2((12))&=&\tdeux,&&\Psi_2((21))&=&\tun\tun,&&\Psi_3(123))&=&\ttroisdeux,&&\Psi_3((132))&=&\ttroisun,\\
\Psi_3(213))&=&\ptroisun,&&\Psi_3((231))&=&\tdeux\tun,&&\Psi_3(312))&=&\tun \tdeux,&&\Psi_3((321))&=&\tun\tun\tun.
\end{array}$$

We define several bijections on $\PP$:

\begin{defi}
Let $P=(P,\leq_h,\leq_r)\in \PP$. We put:
$$\left\{\begin{array}{rcl}
\iota(P)&=&(P,\leq_r,\leq_h),\\
\alpha(P)&=&(P,\geq_h,\leq_r),\\
\beta(P)&=&(P,\leq_h,\geq_r),\\
\gamma(P)&=&(P,\geq_h,\geq_r).
\end{array}\right.$$
\end{defi}

{\bf Remarks.} \begin{enumerate}
\item Graphically:
\begin{itemize}
\item A Hasse graph of $\alpha(P)$ is obtained from a Hasse graph of $P$ by a horizontal symmetry.
\item A Hasse graph of $\beta(P)$ is obtained from a Hasse graph of $P$ by a vertical symmetry.
\item A Hasse graph of $\gamma(P)$ is obtained from a Hasse graph of $P$ by a rotation of angle $\pi$.
\end{itemize}
\item These bijections generate a group of permutations of $\PP$ of cardinality $8$. It is described by the following array:
$$\begin{array}{c||c|c|c|c|c|c|c}
\circ&\alpha&\beta&\gamma&\iota&\iota\circ \alpha&\iota\circ \beta&\iota\circ \gamma\\
\hline \hline \alpha&Id&\gamma&\beta&\iota\circ \beta&\iota \circ \gamma&\iota&\iota \circ \alpha\\
\hline \beta&\gamma&Id&\alpha&\iota \circ \alpha&\iota&\iota \circ \gamma&\iota \circ \beta\\
\hline \gamma&\beta&\alpha&Id&\iota \circ \gamma&\iota \circ \beta&\iota \circ \alpha&\iota\\
\hline \iota&\iota \circ \alpha&\iota \circ \beta&\iota\circ \gamma&Id&\alpha&\beta&\gamma\\
\hline \iota \circ \alpha&\iota&\iota \circ \gamma&\iota\circ \beta&\beta&\gamma&Id&\alpha\\
\hline \iota \circ \beta&\iota \circ \gamma&\iota&\iota \circ \alpha&\alpha&Id&\gamma&\beta\\
\hline \iota \circ \gamma&\iota\circ \beta&\iota \circ \alpha&\iota&\gamma&\beta&\alpha&Id
\end{array}$$
This is a dihedral group $D_4$.
\end{enumerate}

\subsection{Algebraic structures on plane posets}

Two products are defined on $\PP$. The first is called \emph{composition} in \cite{MR2} and denoted by $\prodg$ in \cite{F1}.
We shall shortly denote it by $m$ in this text.

\begin{defi}
Let $P,Q\in \PP$. 
\begin{enumerate}
\item The double poset $PQ=m(P\otimes Q)$ is defined as follows:
\begin{itemize}
\item $PQ=P\sqcup Q$ as a set, and $P,Q$ are plane subposets of $PQ$.
\item if $x\in P$ and $y\in Q$, then $x \leq_r y$ in $PQ$.
\end{itemize}
\item The double poset $P\prodh Q$ is defined as follows:
\begin{itemize}
\item $P\prodh Q=P\sqcup Q$ as a set, and $P,Q$ are plane subposets of $P\prodh Q$.
\item if $x\in P$ and $y\in Q$, then $x \leq_h y$ in $P\prodh Q$.
\end{itemize}
\end{enumerate}\end{defi}

{\bf Examples.}\begin{enumerate}
\item The Hasse graph of $PQ$ is  the concatenation of the Hasse graphs of $P$ and $Q$. 
\item Here are examples for $\prodh$: $\tun \prodh \tdeux=\ttroisdeux,\: \tdeux \prodh \tun=\ttroisdeux,\: \tun \prodh \tun \tun=\ttroisun,\: \tun \tun \prodh \tun=\ptroisun$.\\
\end{enumerate}

The vector space generated by $\PP$ is denoted by $\h_{\PP}$. These two products are linearly extended to $\h_{\PP}$;
then $(\h_{\PP},m)$ and $(\h_{\PP},\prodh)$ are two associative, unitary algebras, sharing the same unit $1$, which is the empty plane poset.
Moreover, they are both graded by the cardinality of plane posets. They are free algebras, as implied that the following theorem, proved in \cite{F1}:

\begin{theo}\begin{enumerate}
\item \begin{enumerate}
\item Let $P$ be a nonempty plane poset. We shall say that $P$ is \emph{$h$-irreducible} if for all $Q, \in \PP$, $P=QR$ implies that $Q=1$ or $R=1$.
\item Any plane poset $P$ can be uniquely written as $P=P_1\ldots P_k$, where $P_1,\ldots,P_k$ are $h$-irreducible.
We shall say that $P_1,\ldots,P_k$ are the \emph{$h$-irreducible components} of $P$.
\end{enumerate}
\item \begin{enumerate}
\item Let $P$ be a nonempty plane poset. We shall say that $P$ is \emph{$r$-irreducible} if for all $Q, \in \PP$, $P=Q\prodh R$ implies that $Q=1$ or $R=1$.
\item Any plane poset $P$ can be uniquely written as $P=P_1\prodh \ldots \prodh P_k$, where $P_1,\ldots,P_k$ are $r$-irreducible.
We shall say that $P_1,\ldots,P_k$ are the \emph{$r$-irreducible components} of $P$.
\end{enumerate}
\end{enumerate}\end{theo}

{\bf Remark}. The Hasse graphs of the $h$-irreducible components of $H$ are the connected components of the Hasse graph of $(P,\leq_h)$,
whereas the Hasse graphs of the $r$-irreducible components of $H$ are the connected components of the Hasse graph of $(P,\leq_r)$.

\begin{defi}\label{7}
Let $P=(P,\leq_h,\leq_r)$ be a plane poset, and let $I \subseteq P$. 
\begin{enumerate}
\item We shall say that $I$ is a \emph{$h$-ideal} of $P$, if, for all $x,y \in P$: $$(x\in I, \:x\leq_h y)\Longrightarrow (y\in I).$$
\item We shall say that $I$ is a \emph{$r$-ideal} of $P$, if, for all $x,y \in P$: $$(x\in I, \:x\leq_r y)\Longrightarrow (y\in I).$$
\item We shall say that $I$ is a \emph{biideal} of $P$ if it both an $h$-ideal and a $r$-ideal. Equivalently, $I$ is a biideal of $P$ if, for all $x,y \in P$:
$$(x\in I, \:x\leq y)\Longrightarrow (y\in I).$$
\end{enumerate}\end{defi}

The following proposition is proved in \cite{F1} (proposition 29):

\begin{prop}
$\h_{\PP}$ is given a coassociative, counitary coproduct in the following way: for any plane poset $P$,
$$\Delta(P)=\sum_{\mbox{\scriptsize $I$ $h$-ideal of $P$}} (P\setminus I)\otimes I.$$
Moreover, $(\h_{\PP},m,\Delta)$ is a Hopf algebra, and $(\h_{\PP},\prodh,\Delta)$ is an infinitesimal Hopf algebra \cite{Loday}, both graded 
by the cardinality of the plane posets. In other words, using Sweedler's notations $\Delta(x)=\sum x^{(1)}\otimes x^{(2)}$, for all $x,y\in \h_{\PP}$:
$$\left\{\begin{array}{rcl}
\Delta(xy)&=&\displaystyle \sum x^{(1)}y^{(1)}\otimes x^{(2)}y^{(2)},\\[2mm]
\Delta(x \prodh y)&=&\displaystyle \sum x\prodh y^{(1)} \otimes y^{(2)}+\sum x^{(1)}\otimes x^{(2)} \prodh y-x \otimes y.
\end{array}\right.$$
\end{prop}

{\bf Remarks.} The following compatibilities are satisfied:
\begin{enumerate}
\item For all $P,Q \in \PP$:
$$\left\{\begin{array}{rclccrcl}
\iota(PQ)&=&\iota(P)\prodh \iota(Q),&&\iota(P\prodh Q)&=&\iota(P)\iota(Q),\\
\alpha(PQ)&=&\alpha(P)\alpha(Q),&&\alpha(P\prodh Q)&=&\alpha(Q)\prodh \alpha(P),\\
\beta(PQ)&=&\beta(Q)\beta(P),&&\beta(P\prodh Q)&=&\beta(P)\prodh \beta(Q),\\
\gamma(PQ)&=&\gamma(Q)\gamma(P),&&\gamma(P\prodh Q)&=&\gamma(Q)\prodh \gamma(P).
\end{array}\right.$$
\item Moreover, $\Delta \circ \alpha=(\alpha \otimes \alpha)\circ \Delta^{op}$, $\Delta \circ \beta=(\beta \otimes \beta)\circ \Delta$, and
$\Delta \circ \gamma=(\gamma \otimes \gamma)\circ \Delta^{op}$.
\end{enumerate}

\subsection{Pairings}

We also defined two pairings on $\h_{\PP}$, using the notion of \emph{pictures}:

\begin{defi}
Let $P,Q$ be two elements of $\PP$. 
\begin{enumerate}
\item We denote by $S(P,Q)$ the set of bijections $\sigma:P\longrightarrow Q$ such that, for all $i,j \in P$:
\begin{itemize}
\item ($i\leq_h j$ in $P$) $\Longrightarrow$ ($\sigma(i) \leq_r \sigma(j)$ in $Q$).
\item ($\sigma(i)\leq_h \sigma(j)$ in $Q$) $\Longrightarrow$ ($i \leq_r j$ in $P$).
\end{itemize}
\item We denote by $S'(P,Q)$ the set of bijections $\sigma:P\longrightarrow Q$ such that, for all $i,j \in P$:
\begin{itemize}
\item ($i\leq_h j$ in $P$) $\Longrightarrow$ ($\sigma(i) \leq \sigma(j)$ in $Q$).
\item ($\sigma(i)\leq_h \sigma(j)$ in $Q$) $\Longrightarrow$ ($i \leq j$ in $P$).
\end{itemize}\end{enumerate}\end{defi}

The following theorem is proved in \cite{F1,F2,MR2}: 

\begin{theo}
We define two pairings:
\begin{eqnarray*}
\langle-,-\rangle&:&\left\{\begin{array}{rcl}
\h_{\PP}\otimes \h_{\PP}&\longrightarrow&K\\
P\otimes Q&\longrightarrow&\langle P,Q\rangle=Card(S(P,Q)), 
\end{array}\right.\\ \\
\langle-,-\rangle'&:&\left\{\begin{array}{rcl}
\h_{\PP}\otimes \h_{\PP}&\longrightarrow&K\\
P\otimes Q&\longrightarrow&\langle P,Q\rangle'=Card(S'(P,Q)).
\end{array}\right. \end{eqnarray*}
They are both homogeneous, symmetric, nondegenerate Hopf pairings on the Hopf algebra $\h_{\PP}=(\h_{\PP},m,\Delta)$. 
\end{theo}

\subsection{Morphism to free quasi-symmetric functions}

We here briefly recall the construction of the Hopf algebra $\FQSym$ of free quasi-symmetric functions, also called the Malvenuto-Reutenauer Hopf algebra
\cite{DHT,MR1}. As a vector space, a basis of $\FQSym$ is given by the disjoint union of the symmetric groups $\S_n$, for all $n \geq 0$.
 By convention, the unique element of $\S_0$ is denoted by $1$. The product of $\FQSym$ is given, for $\sigma \in \S_k$, $\tau \in \S_l$, by:
$$\sigma\tau=\sum_{\epsilon \in Sh(k,l)} (\sigma \otimes \tau) \circ \epsilon,$$
where $Sh(k,l)$ is the set of $(k,l)$-shuffles, that is to say permutations $\epsilon \in \S_{k+l}$ such that $\epsilon^{-1}(1)<\ldots <\epsilon^{-1}(k)$
and $\epsilon^{-1}(k+1)<\ldots<\epsilon^{-1}(k+l)$.
In other words, the product of $\sigma$ and $\tau$ is given by shifting the letters of the word
representing $\tau$ by $k$, and then summing all the possible shufflings of this word and of the word representing $\sigma$.
For example:
\begin{eqnarray*}
(123)(21)&=&(12354)+(12534)+(15234)+(51234)+(12543)\\
&&+(15243)+(51243)+(15423)+(51423)+(54123).
\end{eqnarray*}

Let $\sigma \in \Sigma_n$. For all $0\leq k \leq n$, there exists a unique triple 
$\left(\sigma_1^{(k)},\sigma_2^{(k)},\zeta_k\right)\in \S_k \times \S_{n-k} \times Sh(k,n-k)$
such that $\sigma=\zeta_k^{-1} \circ \left(\sigma_1^{(k)} \otimes \sigma_2^{(k)}\right)$. The coproduct of $\FQSym$ is then defined by:
$$\Delta(\sigma)=\sum_{k=0}^n \sigma_1^{(k)} \otimes \sigma_2^{(k)}.$$
For example:
\begin{eqnarray*}
\Delta((43125))&=&1\otimes (43125)+(1) \otimes (3124)+(21) \otimes (123)\\
&&+(321)\otimes (12)+(4312) \otimes (1)+(43125) \otimes 1.
\end{eqnarray*}
Note that $\sigma_1^{(k)}$ and $\sigma_2^{(k)}$ are obtained by cutting the word representing $\sigma$ between the $k$-th and the $(k+1)$-th letter,
and then {\it standardizing} the two obtained words, that is to say applying to their letters the unique increasing bijection to $\{1,\ldots,k\}$ or $\{1,\ldots,n-k\}$.
Moreover, $\FQSym$ has a nondegenerate, homogeneous, Hopf pairing defined by $\langle \sigma,\tau\rangle=\delta_{\sigma,\tau^{-1}}$
for all permutations $\sigma$ and $\tau$.

\begin{defi}\label{11} \cite{FU,MR2}
\begin{enumerate}
\item Let $P=(P,\leq_h,\leq_r)$ a plane poset. Let $x_1<\ldots< x_n$ be the elements of $P$, which is totally ordered.
A \emph{linear extension} of $P$ is a permutation $\sigma \in \S_n$ such that, for all $i,j \in \{1,\ldots,n\}$:
$$(x_i\leq_h x_j)\Longrightarrow (\sigma^{-1}(i)<\sigma^{-1}(j)).$$
The set of linear extensions of $P$ will be denoted by $S_P$.
\item The following map is an isomorphism of Hopf algebras:
$$\Theta:\left\{\begin{array}{rcl}
\h_{\PP}&\longrightarrow&\FQSym\\
P\in \PP&\longrightarrow&\displaystyle \sum_{\sigma \in S_P}\sigma.
\end{array}\right.$$
moreover, for all $x,y\in \h_{\PP}$, $\langle \Theta(x),\Theta(y)\rangle'=\langle x,y\rangle$.
\end{enumerate}\end{defi}

\section{Deformation of the products}

\subsection{Construction}

\begin{defi}
Let $P\in \PP$ and $X,Y \subseteq P$. We put:
\begin{enumerate}
\item $h_X^Y=\sharp\{(x,y) \in X\times Y\:/\:x \leq_h y \mbox{ in }P\}$.
\item $r_X^Y=\sharp\{(x,y) \in X\times Y\:/\:x \leq_r y \mbox{ in }P\}$.
\end{enumerate}\end{defi}

\begin{lemma}\label{13}
Let $X$ and $Y$ be disjoint parts of a plane poset $P$. Then:
$$h_X^Y+h_Y^X+r_X^Y+r_Y^X=|X||Y|.$$
\end{lemma}

\begin{proof} Indeed, $h_X^Y+h_Y^X+r_X^Y+r_Y^X=\sharp\{(x,y)\in X \times Y\mid x<y \mbox{ or }x>y\}=|X \times Y|$. \end{proof}\\

{\bf Remark.} If it more generally possible to prove that for any part $X$ and $Y$ of a plane poset:
$$h_X^Y+h_Y^X+r_X^Y+r_Y^X=3|X \cap Y|^2+|X||Y|.$$

\begin{theo}\label{14}
Let $q=(q_1,q_2,q_3,q_4) \in K^4$. We consider the following map:
$$m_q: \left\{ \begin{array}{rcl}
\h_{\PP}\otimes \h_{\PP}&\longrightarrow &\h_{\PP}\\
P\otimes Q&\longrightarrow &\displaystyle \sum_{\substack{(R,I)\in \PP^2\\I\subseteq R,\:R\setminus I=P,\:I=Q}}
q_1^{h_{R\setminus I}^I}q_2^{h_I^{R\setminus I}}q_3^{r_{R\setminus I}^I}q_4^{r_I^{R\setminus I}} R,
\end{array}\right.$$
where $P,Q \in \PP$. Then $(\h_{\PP},m_q)$ is an associative algebra, and its unit is the empty plane poset $1$.
\end{theo}

\begin{proof} Let $P,Q,R \in \PP$. We put:
$$\left\{\begin{array}{rcl}
P_1&=&(m_q \otimes Id) \circ m_q(P \otimes Q \otimes R),\\
P_2&=&(Id \otimes m_q) \circ m_q(P \otimes Q \otimes R).
\end{array}\right.$$
Then:
$$\left\{\begin{array}{rcl}
P_1&=&\displaystyle \sum_{(R_1,I_1,R_2,I_2)\in E_1}q_1^{h_{R_1\setminus I_1}^{I_1}+h_{R_2\setminus I_2}^{I_2}}
q_2^{h^{R_1\setminus I_1}_{I_1}+h^{R_2\setminus I_2}_{I_2}}q_3^{r_{R_1\setminus I_1}^{I_1}+r_{R_2\setminus I_2}^{I_2}}
q_4^{r^{R_1\setminus I_1}_{I_1}+r^{R_2\setminus I_2}_{I_2}} R_2,\\[2mm]
P_2&=&\displaystyle \sum_{(R_1,I_1,R_2,I_2)\in E_2}q_1^{h_{R_1\setminus I_1}^{I_1}+h_{R_2\setminus I_2}^{I_2}}
q_2^{h^{R_1\setminus I_1}_{I_1}+h^{R_2\setminus I_2}_{I_2}}q_3^{r_{R_1\setminus I_1}^{I_1}+r_{R_2\setminus I_2}^{I_2}}
q_4^{r^{R_1\setminus I_1}_{I_1}+r^{R_2\setminus I_2}_{I_2}} R_2,
\end{array}\right.$$
With:
\begin{eqnarray*}
E_1&=&\{(R_1,I_1,R_2,I_2)\in \PP^4\:/\:I_1\subseteq R_1,\: I_1=Q,\: R_1\setminus I_1=P,\:I_2\subseteq R_2,\: I_2=R,\:R_2\setminus I_2=R_1\},\\
E_2&=&\{(R_1,I_1,R_2,I_2)\in \PP^4\:/\:I_1\subseteq R_1,\: I_1=R,\: R_1\setminus I_1=Q,\:I_2\subseteq R_2,\: I_2=R_1,\:R_2\setminus I_2=P\}.
\end{eqnarray*}
We shall also consider:
$$E=\{(R,J_1,J_2,J_3)\in \PP^4\:/\:R=J_1 \sqcup J_2 \sqcup J_3,\:J_1=P,\:J_2=Q,\:J_3=R\}.$$

{\it First step.} Let us consider the following maps:
\begin{eqnarray*}
&\phi:& \left\{ \begin{array}{rcl}
E_1&\longrightarrow & E\\
(R_1,I_1,R_2,I_2)&\longrightarrow &(R_2,R_1\setminus I_1,I_1,I_2),
\end{array}\right. \\
&\phi':& \left\{ \begin{array}{rcl}
E&\longrightarrow & E_1\\
(R,J_1,J_2,J_3)&\longrightarrow &(J_1\sqcup J_2, J_2,R,J_3).
\end{array}\right. \end{eqnarray*}
By definition of $E$ and $E_1$, these maps are well-defined, and an easy computation shows that
$\phi \circ \phi'=Id_E$ and $\phi' \circ \phi=Id_{E_1}$, so $\phi$ is a bijection.
Let $(R_1,I_1,R_2,I_2)\in E_1$. We put $\phi(R_1,I_1,R_2,I_2)=(R,J_1,J_2,J_3)$. Then:
$$h_{R_1\setminus I_1}^{I_1}+h_{R_2\setminus I_2}^{I_2}=h_{J_1}^{J_2}
+h_{J_1 \sqcup J_2}^{J_3}=h_{J_1}^{J_2}+h_{J_1}^{J_3}+h_{J_2}^{J_3}.$$
Similar computations finally give:
$$P_1=\sum_{(R,J_1,J_2,J_3)\in E}q_1^{h_{J_1}^{J_2}+h_{J_1}^{J_3}+h_{J_2}^{J_3}}
q_2^{h^{J_1}_{J_2}+h^{J_1}_{J_3}+h^{J_2}_{J_3}}
q_3^{r_{J_1}^{J_2}+r_{J_1}^{J_3}+r_{J_2}^{J_3}}q_4^{r^{J_1}_{J_2}+r^{J_1}_{J_3}+r^{J_2}_{J_3}} R.$$

{\it Second step.} Let us consider the following maps:
\begin{eqnarray*}
&\psi:& \left\{ \begin{array}{rcl}
E_2&\longrightarrow & E\\
(R_1,I_1,R_2,I_2)&\longrightarrow &(R_2,R_1\setminus R_1,R_1\setminus I_1,I_1),
\end{array}\right.\\
&\psi':& \left\{ \begin{array}{rcl}
E&\longrightarrow & E_2\\
(R,J_1,J_2,J_3)&\longrightarrow &(J_2\sqcup J_3, J_3,R,J_2 \sqcup J_3).
\end{array}\right. \end{eqnarray*}
By definition of $E$ and $E_2$, these maps are well-defined, and a simple computation shows that $\psi \circ \psi'=Id_E$ and $\psi' \circ \psi=Id_{E_2}$,
so $\psi$ is a bijection. Let $(R_1,I_1,R_2,I_2)\in E_2$. We put $\psi(R_1,I_1,R_2,I_2)=(R,J_1,J_2,J_3)$. Then:
$$h_{R_1\setminus I_1}^{I_1}+h_{R_2\setminus I_2}^{I_2}=h_{J_2}^{J_3}+h_{J_1}^{J_2 \sqcup J_3}=h_{J_1}^{J_2}+h_{J_1}^{J_3}+h_{J_2}^{J_3}.$$
Similar computations finally give:
$$P_2=\sum_{(R,J_1,J_2,J_3)\in E}q_1^{h_{J_1}^{J_2}+h_{J_1}^{J_3}+h_{J_2}^{J_3}}q_2^{h^{J_1}_{J_2}+h^{J_1}_{J_3}+h^{J_2}_{J_3}}
q_3^{r_{J_1}^{J_2}+r_{J_1}^{J_3}+r_{J_2}^{J_3}}q_4^{r^{J_1}_{J_2}+r^{J_1}_{J_3}+r^{J_2}_{J_3}} R.$$
So $m_q$ is associative.\\

{\it Last step.} Let $P \in \PP$. Then:
$$P.1=\sum_{\substack{(R,I)\in \PP^2\\I\subseteq R,\:R\setminus I=P,\:I=1}}
q_1^{h_{R_I}^I}q_2^{h_I^{R\setminus I}}q_3^{r_{R\setminus I}^I}q_4^{r_I^{R\setminus I}} R=q_1^{h_P^1}q_2^{h_1^P}q_3^{r_P^1}q_4^{r_1^P} P=P.$$
Similarly, for all $Q\in \PP$, $1.Q=Q$. \end{proof}\\

{\bf Examples.}
\begin{eqnarray*}
m_q(\tun \otimes \tun)&=&q_3q_4 \tun\tun+q_1q_2 \tdeux,\\
m_q(\tun \otimes \tdeux)&=&q_3^2 \tun\tdeux+q_4^2 \tdeux\tun+q_2(q_3+q_4)\ttroisun+q_1(q_3+q_4)\ptroisun+(q_1^2+q_1q_2+q_2^2) \ttroisdeux,\\
m_q(\tun \otimes \tun\tun)&=&q_1^2 \ttroisun+q_2^2 \ptroisun+(q_1+q_2)q_4\tun\tdeux+(q_1+q_2)q_3\tdeux\tun+(q_3^2+q_3q_4+q_4^2)\tun\tun\tun,\\
m_q(\tdeux \otimes \tun)&=&q_3^2 \tun\tdeux+q_4^2 \tdeux\tun+q_1(q_3+q_4)\ttroisun+q_2(q_3+q_4)\ptroisun+(q_1^2+q_1q_2+q_2^2) \ttroisdeux,\\
m_q(\tun\tun \otimes \tun)&=&q_2^2 \ttroisun+q_1^2 \ptroisun+(q_1+q_2)q_4\tun\tdeux+(q_1+q_2)q_3\tdeux\tun+(q_3^2+q_3q_4+q_4^2)\tun\tun\tun.
\end{eqnarray*}

The following result is immediate:

\begin{prop} \label{15}
Let $(q_1,q_2,q_3,q_4)\in K^4$. Then:
$$\left\{\begin{array}{rcl}
m_{(q_1,q_2,q_3,q_4)}^{op}&=&m_{(q_2,q_1,q_4,q_3)},\\[2mm]
m_{(q_1,q_2,q_3,q_4)}\circ (\iota \otimes \iota)&=&\iota \circ m_{(q_3,q_4,q_1,q_2)},\\[2mm]
m_{(q_1,q_2,q_3,q_4)}\circ (\alpha \otimes \alpha)&=&\alpha \circ m_{(q_2,q_1,q_3,q_4)},\\[2mm]
m_{(q_1,q_2,q_3,q_4)}\circ (\beta \otimes \beta)&=&\beta \circ m_{(q_1,q_2,q_4,q_3)},\\[2mm]
m_{(q_1,q_2,q_3,q_4)}\circ (\gamma\otimes \gamma)&=&\gamma \circ m_{(q_2,q_1,q_4,q_3)}.
\end{array}\right.$$ \end{prop}

\subsection{Particular cases}

\begin{lemma} \label{21}
Let $P\in \PP$ and $X\subseteq P$.
\begin{enumerate}
\item $X$ is a $h$-ideal of $P$ if, and only if, $h_X^{P\setminus X}=0$.
\item $X$ is a $r$-ideal of $P$ if, and only if, $r_X^{P\setminus X}=0$.
\item The $h$-irreducible components of $P$ are the $h$-irreducible components of $X$ and $P\setminus X$ if, and only if, 
$h_{P\setminus X}^X=h_X^{P\setminus X}=0$.
\item The $r$-irreducible components of $P$ are the $r$-irreducible components of $X$ and $P\setminus X$ if, and only if, 
$r_{P\setminus X}^X=r_X^{P\setminus X}=0$.
\item $P=X(P\setminus X)$ if, and only if, $h_{P\setminus X}^X=h_X^{P\setminus X}=r_{P\setminus X}^X=0$.
\item $P=X \prodh (P\setminus X)$ if, and only if, $r_{P\setminus X}^X=r_X^{P\setminus X}=h_{P\setminus X}^X=0$.
\end{enumerate} \end{lemma}

\begin{proof} We give the proofs of points $1$, $3$ and $5$. The others are similar.\\

1. $\Longleftarrow$. Let $x \in X$, $y\in P$, such that $x\leq_h y$. As $h_X^{P\setminus X}=0$, $y\notin P\setminus X$, so $y \in X$: $X$ is a $h$-ideal.

$\Longrightarrow$. Then, for all $x\in X$, $y\in P\setminus X$, $x\leq_h y$ is not possible. So $h_X^{P\setminus X}=0$.\\

3. $\Longrightarrow$. Let us put $P=P_1 \ldots P_k$, where $P_1,\ldots,P_k$ are the $h$-irreducible components of $P$.
By hypothesis, $X$ is the disjoint union of certain $P_i$'s, and $P\setminus X$ is the disjoint union of the other $P_i$'s. So $h_{P\setminus X}^X=
h_X^{P\setminus X}=0$. 

$\Longleftarrow$. Let $I$ be a $h$-irreducible component of $P$. If $I \cap X$ and $I\cap (P\setminus X)$, then for any $x \in I\cap X$ 
and any $y \in I\cap (P\setminus X)$, $x$ and $y$ are not comparable for $\leq_h$: contradiction. So $I$ is included in $I$ or in $P\setminus X$,
so is a $h$-irreducible component of $X$ or $P\setminus X$. \\

5. $\Longrightarrow$. If $x \in X$ and $y \in P\setminus X$, then $x <_r y$, so $h_{P\setminus X}^X=h_X^{P\setminus X}=r_{P\setminus X}^X=0$. 

$\Longleftarrow$. If $x \in X$ and $y\in P\setminus X$, by hypothesis we do not have $x<_h y$, $x >_h y$ nor $x >_r y$, so $x<_r y$.
Hence, $P=X(P\setminus X)$.  \end{proof}

\begin{prop} \label{16}\begin{enumerate}
\item Let us assume that $q_3=q_4=0$. Let $P=P_1\prodh\ldots \prodh P_k$ and $P'=P_{k+1}\prodh \ldots \prodh P_{k+l}$ be two plane posets,
decomposed into their $r$-connected components. Then:
$$m_q(P\otimes P')=\sum_{\sigma\in Sh(k,l)} \prod_{1\leq i\leq k<j\leq k+l} Q_{i,j}(\sigma)^{|P_i||P_j|} P_{\sigma(1)}\prodh \ldots
\prodh P_{\sigma(k+l)},$$
where $Q_{i,j}(\sigma)=q_1$ if $\sigma^{-1}(i)<\sigma^{-1}(j)$ and $q_2$ if $\sigma^{-1}(i)>\sigma^{-1}(j)$.
\item Let us assume that $q_1=q_2=0$. Let $P=P_1\ldots P_k$ and $P'=P_{k+1} \ldots P_{k+l}$ be two plane posets,
decomposed into their $r$-connected components. Then:
$$m_q(P\otimes P')=\sum_{\sigma\in Sh(k,l)} \prod_{1\leq i\leq k<j\leq k+l} Q'_{i,j}(\sigma)^{|P_i||P_j|} P_{\sigma(1)} \ldots
 P_{\sigma(k+l)},$$
where $Q'_{i,j}(\sigma)=q_3$ if $\sigma^{-1}(i)<\sigma^{-1}(j)$ and $q_4$ if $\sigma^{-1}(i)>\sigma^{-1}(j)$.
\end{enumerate}\end{prop}

\begin{proof} Let us prove the first point; the proof of the second point is similar.
Let us consider a plane poset $R$ such that the coefficient of $R$ in $m_q(P \otimes P')$ is not zero. 
So there exists $I \subseteq R$ such that $R\setminus I=P$ and $I=P'$. Moreover, as $q_3=q_4=0$, $r_{R\setminus I}^I=r_I^{R\setminus I}=0$.
By lemma \ref{21}-4, the $r$-connected components of $R$ are the $r$-components of $R\setminus I$ and $I$.
As a consequence, there exists a $(k,l)$-shuffle $\sigma$, such that $R=P_{\sigma(1)}\prodh \ldots \prodh P_{\sigma(k+l)}$.
Then $h_{R\setminus I}^I$ is the sum of $|P_i||P_j|$, where $1\leq i\leq k<j\leq k+l$, such that $\sigma^{-1}(i)<\sigma^{-1}(j)$; $h_I^{R\setminus I}$ 
is the sum of $|P_i||P_j|$, where $1\leq i\leq k<j\leq k+l$, such that $\sigma^{-1}(i)>\sigma^{-1}(j)$. This immediately implies the announced result.  \end{proof}\\

{\bf Remarks.}\begin{enumerate}
\item  The first point implies that:
\begin{itemize}
\item  $m_{(q_1,0,0,0)}(P\otimes P')=q_1^{|P||P'|} P \prodh P'$. In particular, $m_{(1,0,0,0)}=\prodh$.
\item $m_{(0,q_2,0,0)}(P\otimes P')=q_2^{|P||P'|} P' \prodh P$. In particular, $m_{(0,1,0,0)}=\prodh^{op}$.
\item $m_{(0,0,q_3,0)}(P\otimes P')=q_3^{|P||P'|} PP'$. In particular, $m_{(1,0,0,0)}=m$.
\item $m_{(0,0,0,q_4)}(P\otimes P')=q_4^{|P||P'|} P'P$. In particular, $m_{(1,0,0,0)}=m^{op}$.
\end{itemize}
\item It is possible to define $m_q$ on the space of double posets. The same arguments prove that it is still associative.
However, $m_{(1,0,0,0)}$ is not equal to $\prodh$ on $\h_{\PP}$ and $m_{(0,0,1,0)}$ is not equal to $m$; for example,
if we denote by $\wp_2$  the double poset with two elements $x,y$, $x$ and $y$ being not comparable for $\leq_h$ and $\leq_r$:
$$m_{(1,0,0,0)}(\tun \otimes \tun)=\tdeux+2\wp_2,\hspace{1cm} m_{(0,0,1,0)}(\tun \otimes \tun)=\tun \tun+2\wp_2.$$ 
\end{enumerate}

\subsection{Subalgebras and quotients}

These two particular families of plane posets are used in \cite{F1,F2}:

\begin{defi}
Let $P \in \PP$.
\begin{enumerate}
\item We shall say that $P$ is a \emph{plane forest} if it does not contain $\ptroisun$ as a plane subposet.
The set of plane forests is denoted by $\PF$.
\item We shall say that $P$ is \emph{WN} ("without N") if it does not contain $\pquatrecinq$ nor $\pquatresix$.
The set of WN posets is denoted by $\WNP$.
\end{enumerate}\end{defi}

{\bf Examples.} A plane poset is a plane forest if, and only if, its Hasse graph is a rooted forest.
\begin{eqnarray*}
\PF(0)&=&\{1\},\\
\PF(1)&=&\{\tun\},\\
\PF(2)&=&\{\tun\tun,\tdeux\},\\
\PF(3)&=&\{\tun\tun\tun,\tun\tdeux,\tdeux\tun,\ttroisun,\ttroisdeux\},\\
\PF(4)&=&\left\{\tun\tun\tun\tun,\tun\tun\tdeux,\tun\tdeux\tun,\tdeux\tun\tun,\tun\ttroisun,\ttroisun\tun,
\tun\ttroisdeux,\ttroisdeux\tun,\tdeux\tdeux,\tquatreun,\tquatredeux,\tquatretrois,\tquatrequatre,\tquatrecinq\right\},\\ \\
\WNP(0)&=&\{1\},\\
\WNP(1)&=&\{\tun\},\\
\WNP(2)&=&\{\tun\tun,\tdeux\},\\
\WNP(3)&=&\{\tun\tun\tun,\tun\tdeux,\tdeux\tun,\ttroisun,\ttroisdeux,\ptroisun\},\\
\WNP(4)&=&\left\{\begin{array}{c}
\tun\tun\tun\tun,\tun\tun\tdeux,\tun\tdeux\tun,\tdeux\tun\tun,\tun\ttroisun,\ttroisun\tun,\tun\ttroisdeux,\ttroisdeux\tun,\tun\ptroisun,\ptroisun\tun,\tdeux\tdeux,\\
\tquatreun,\tquatredeux,\tquatretrois,\tquatrequatre,\tquatrecinq,\pquatreun,\pquatredeux,\pquatretrois,\pquatrequatre,\pquatresept,\pquatrehuit
\end{array}\right\}.\end{eqnarray*}

\begin{defi}
We denote by:
\begin{itemize}
\item $\h_{\WNP}$ the subspace of $\h_{\PP}$ generated by WN plane posets.
\item  $\h_{\PF}$ the subspace of $\h_{\PP}$ generated by plane forests.
\item $I_{\WNP}$ the subspace of $\h_{\PP}$ generated by plane posets which are not WN.
\item  $I_{\PF}$ the subspace of $\h_{\PP}$ generated by plane posets which are not plane forests.
\end{itemize} \end{defi}

Note that $\h_{\WNP}$ and $\h_{\PF}$ are naturally identified with $\h_{\PP}/I_{\WNP}$ and $\h_{\PP}/I_{\PF}$.

\begin{prop} Let $q=(q_1,q_2,q_3,q_4)\in K^4$.
\begin{enumerate}
\item $\h_{\WNP}$ is a subalgebra of $(\h_{\PP},m_q)$ if and only if, $q_1=q_2=0$ or $q_3=q_4=0$.
\item $\h_{\PF}$ is a subalgebra of $(\h_{\PP},m_q)$ if and only if, $q_1=q_2=0$.
\item $I_{\WNP}$ and $I_{\PF}$ are ideals of $(\h_{\PP},m_q)$.
\end{enumerate}\end{prop}

\begin{proof} $1$. $\Longleftarrow$.  We use the notations of proposition \ref{16}-1. If $q_3=q_4=0$, let us consider two WN posets $P$ and $P'$.  
then the $P_i$'s are also WN, so for any $\sigma \in \S_{k+l}$, $P_{\sigma^{-1}(1)}\prodh \ldots \prodh P_{\sigma^{-1}(k+l)}$ is WN. 
As a conclusion, $m_q(P\otimes P')\in \h_{\WNP}$. The proof is similar if $q_1=q_2=0$, using proposition \ref{16}-2.\\

 $1$. $\Longrightarrow$. Let us consider the coefficients of $\pquatrecinq$ and $\pquatresix$ in certain products. We obtain:
 $$\begin{array}{c|c|c}
&\pquatrecinq&\pquatresix\\
\hline m_q(\ptroisun\otimes \tun)&q_1q_3^2&q_1q_4^2\\
\hline m_q(\tun \otimes \ptroisun)&q_2q_4^2&q_2q_3^2.
\end{array}$$
 If $\h_{\WNP}$ is a subalgebra of $(\h_{\PP},m_q)$, then these four coefficients are zero, so, from the first row, $q_1=0$ or $q_3=q_4=0$
 and from the second row, $q_2=0$ or $q_3=q_4=0$. As a conclusion, $q_1=q_2=0$ or $q_3=q_4=0$.\\
 
$2$. $\Longleftarrow$. We use the notations of proposition \ref{16}-2. If $q_1=q_2=0$, let us consider two plane forests $P$ and $P'$.  
Then the $P_i$'s are plane trees, so for any $\sigma \in \S_{k+l}$, $P_{\sigma^{-1}(1)}\prodh \ldots \prodh P_{\sigma^{-1}(k+l)}$ is a plane forest. 
As a conclusion, $m_q(P\otimes P')\in \h_{\PF}$.\\

 $2$. $\Longrightarrow$. Let us consider the coefficients of $\ptroisun$ in certain products. We obtain:
 $$\begin{array}{c|c}
&\ptroisun\\
\hline m_q(\tun\otimes \tun\tun)&q_2^2\\
\hline m_q(\tun\tun \otimes \tun)&q_1^2.
\end{array}$$
 If $\h_{\PF}$ is a subalgebra of $(\h_{\PP},m_q)$, then $q_1=q_2=0$.\\
 
$3$. Let $P$ and $P'$ be two plane posets such that $P$ or $P'$ is not WN. Let us consider a plane poset $R$ such that the coefficient
of $R$ in $m_q(P \otimes P')$ is not zero. There exists $I \subseteq R$, such that $R\setminus I=P$ and $I=P'$.
As $P$ or $P'$ is not WN, $I$ or $R\setminus I$ contains $\pquatrecinq$ or $\pquatresix$, so $R$ contains $\pquatrecinq$ or $\pquatresix$:
$R$ is not WN. So $m_q(P\otimes P') \subseteq I_{\WNP}$.
The proof is similar for $I_{\PF}$, using $\ptroisun$ instead of $\pquatrecinq$ and $\pquatresix$. \end{proof}

\section{Dual coproducts}

\subsection{Constructions}

Dually, we give $\h_{\PP}$ a family of coproducts $\Delta_q$, for $q\in K^4$, defined  for all $P\in \PP$ by:
$$\Delta_q(P)=\sum_{I\subseteq P}q_1^{h_{P\setminus I}^I}q_2^{h_I^{P\setminus I}}
q_3^{r_{P\setminus I}^I}q_4^{r_I^{P\setminus I}} (P\setminus I)\otimes I.$$
These coproducts are coassociative; their common counit is given by:
$$\varepsilon: \left\{ \begin{array}{rcl}
\h_{\PP}&\longrightarrow & K\\
P\in \PP&\longrightarrow &\delta_{1,P}.
\end{array}\right.$$

{\bf Examples.} We put, for all $P \in \PP$, nonempty, $\tdelta_q(P)=\Delta(P)-P\otimes 1-1\otimes P$.
\begin{eqnarray*}
\tdelta_q(\tdeux)&=&(q_1+q_2) \tun \otimes \tun,\\
\tdelta_q(\tun\tun)&=&(q_3+q_4 )\otimes \tun,\\
\tdelta(\ttroisdeux)&=&(q_1^2+q_1q_2+q_2^2) \tun \otimes \tdeux+(q_1^2+q_1q_2+q_2^2)\tdeux \otimes \tun,\\
\tdelta(\ttroisun)&=&q_2(q_3+q_4)\tun \otimes \tdeux+q_1(q_3+q_4)\tdeux \otimes \tun+q_1^2 \tun \otimes \tun \tun+q_2^2 \tun\tun \otimes \tun,\\
\tdelta(\ptroisun)&=&q_1(q_3+q_4)\tun \otimes \tdeux+q_2(q_3+q_4)\tdeux \otimes \tun+q_2^2 \tun \otimes \tun \tun+q_1^2 \tun\tun \otimes \tun,\\
\tdelta(\tdeux \tun)&=&q_4^2\tun \otimes \tdeux+q_3^2\tdeux \otimes \tun+(q_1+q_2)q_3 \tun \otimes \tun \tun+(q_1+q_2)q_4 \tun\tun \otimes \tun,\\
\tdelta(\tun\tdeux)&=&q_3^2\tun \otimes \tdeux+q_4^2\tdeux \otimes \tun+(q_1+q_2)q_4 \tun \otimes \tun \tun+(q_1+q_2)q_3 \tun\tun \otimes \tun,\\
\tdelta(\tun\tun\tun)&=&(q_3^2+q_3q_4+q_4^2)\tun \otimes \tun\tun+(q_3^2+q_3q_4+q_4^2)\tun\tun \otimes \tun.
\end{eqnarray*}

Dualizing proposition \ref{15}:

\begin{prop}
Let $(q_1,q_2,q_3,q_4)\in K^4$. Then:
$$\left\{\begin{array}{rcl}
\Delta_{(q_1,q_2,q_3,q_4)}^{op}&=&\Delta_{(q_2,q_1,q_4,q_3)},\\[2mm]
(\iota \otimes \iota)\circ \Delta_{(q_1,q_2,q_3,q_4)}&=& \Delta_{(q_3,q_4,q_1,q_2)}\circ\iota,\\[2mm]
(\alpha \otimes \alpha)\circ \Delta_{(q_1,q_2,q_3,q_4)}&=& \Delta_{(q_2,q_1,q_3,q_4)}\circ\alpha,\\[2mm]
(\beta \otimes \beta)\circ \Delta_{(q_1,q_2,q_3,q_4)}&=& \Delta_{(q_1,q_2,q_4,q_3)}\circ\beta,\\[2mm]
(\gamma \otimes \gamma)\circ \Delta_{(q_1,q_2,q_3,q_4)}&=& \Delta_{(q_2,q_1,q_4,q_3)}\circ\gamma.
\end{array}\right.$$  \end{prop}

\subsection{Particular cases}

\begin{prop}\label{22}
Let $P\in \PP$. Then:
\begin{enumerate}
\item $\displaystyle \Delta_{(q,q,q,q)}(P)=\sum_{I\subseteq P} q^{|P\setminus I||I|}(P\setminus I)\otimes I$.
\item 
$$\left\{\begin{array}{rcl}
\Delta_{(0,q_2,q_3,q_4)}(P)&=&\displaystyle \sum_{\mbox{\scriptsize $I$ $h$-ideal of $P$}} 
q_2^{h_{P\setminus I}^I}q_3^{r_I^{P\setminus I}}q_4^{r_{P\setminus I}^I}I \otimes (P\setminus I),\\
\Delta_{(q_1,0,q_3,q_4)}(P)&=&\displaystyle \sum_{\mbox{\scriptsize $I$ $h$-ideal of $P$}} 
q_1^{h_{P\setminus I}^I}q_4^{r_I^{P\setminus I}}q_3^{r_{P\setminus I}^I}(P\setminus I) \otimes I.
\end{array}\right.$$
\item $$\left\{\begin{array}{rcl}
\Delta_{(q_1,q_2,0,q_4)}(P)&=&\displaystyle \sum_{\mbox{\scriptsize $I$ $r$-ideal of $P$}}
q_1^{h_I^{P\setminus I}}q_2^{h_{P\setminus I}^I}q_4^{r_{P\setminus I}^I} I \otimes (P\setminus I),\\
\Delta_{(q_1,q_2,q_3,0)}(P)&=&\displaystyle \sum_{\mbox{\scriptsize $I$ $r$-ideal of $P$}} 
q_2^{h_I^{P\setminus I}}q_1^{h_{P\setminus I}^I}q_3^{r_{P\setminus I}^I}(P\setminus I) \otimes I.\\
\end{array}\right.$$
\item  $$\left\{\begin{array}{rcl}
\Delta_{(0,q_2,0,q_4)}(P)&=&\displaystyle \sum_{\mbox{\scriptsize $I$ biideal of $P$}} 
q_2^{h_{P\setminus I}^I}q_4^{r_{P\setminus I}^I}I \otimes (P\setminus I),\\
\Delta_{(q_1,0,q_3,0)}(P)&=&\displaystyle \sum_{\mbox{\scriptsize $I$ biideal of $P$}} 
q_1^{h_{P\setminus I}^I}q_3^{r_{P\setminus I}^I}(P\setminus I) \otimes I.\\
\end{array}\right.$$
\item If $P=P_1\cdots P_k$, where the $P_i$'s are $h$-connected,
$$\Delta_{(0,0,q_3,q_4)}(P)=\sum_{I\subseteq \{1,\cdots,k\}}q_3^{\alpha_P(I)}q_4^{\alpha_P(\{1,\ldots,k\}\setminus I)}
 P_I\otimes P_{\{1,\cdots,k\}\setminus I},$$
with, for all $J=\{j_1,\cdots,j_l\}$, $1\leq j_1<\cdots <j_l\leq k$, $P_J =P_{j_1} \cdots P_{j_l}$, 
and $\alpha_P(J)=\displaystyle \sum_{i\in J,j\notin J,i<j}|P_i||P_j|$.
\item If $P=P_1\prodh\cdots \prodh P_k$, where the $P_i$'s are $r$-connected,
$$\Delta_{(q_1,q_2,0,0)}(P)=\sum_{I\subseteq \{1,\cdots,k\}} q_1^{\beta_P(I)}q_2^{\beta_P(\{1,\ldots,k\}\setminus I)}
P_I^\prodh\otimes P_{\{1,\cdots,k\}\setminus I}^\prodh,$$
with, for all $J=\{j_1,\cdots,j_l\}$, $1\leq j_1<\cdots <j_l\leq k$, $P_J\prodh =P_{j_1}\prodh \cdots \prodh P_{j_l}$, 
and $\beta_P(J)=\displaystyle \sum_{i\in J,j\notin J,i<j}|P_i||P_j|$.

\item  $$\left\{\begin{array}{rcl}
\Delta_{(q,0,0,0)}(P)&=&\displaystyle \sum_{P_1 \prodh P_2=P} q^{|P_1||P_2|}P_1 \otimes P_2,\\
\Delta_{(0,q,0,0)}(P)&=&\displaystyle \sum_{P_1 \prodh P_2=P} q^{|P_1||P_2|}P_2 \otimes P_1,\\
\Delta_{(0,0,q,0)}(P)&=&\displaystyle \sum_{P_1P_2=P} q^{|P_1||P_2|}P_1 \otimes P_2,\\
\Delta_{(0,0,0,q)}(P)&=&\displaystyle \sum_{P_1P_2=P} q^{|P_1||P_2|}P_2 \otimes P_1 .\\
\end{array}\right.$$
\item $\displaystyle \Delta_{(0,0,0,0)}(P)=P\otimes 1+1\otimes P$ if $P\neq 1$.
\end{enumerate} \end{prop}

\begin{proof}  We only prove points 1, 2, 5, 7 and 8; the others are proved in the same way.

1. Immediate, as $h_{\P\setminus I}^I+h_I^{P\setminus I}+r_{\P\setminus I}^I+r_I^{P\setminus I}=|I||P\setminus I|$ by lemma \ref{13}.\\

2. By lemma \ref{21}:
$$\Delta_{(q_1,0,q_3,q_4)}(P)=\sum_{\substack{I\subseteq P\\h_{P\setminus I}^I=0}}
q_1^{h_{P\setminus I}^I}q_3^{r_{P\setminus I}^I}q_4^{r_I^{P\setminus I}}(R\setminus I)\otimes I
=\sum_{\mbox{\scriptsize $P\setminus I$ $h$-ideal of $P$}} 
q_1^{h_{P\setminus I}^I}q_3^{r_{P\setminus I}^I}q_4^{r_I^{P\setminus I}}(P\setminus I) \otimes I.$$
As $\Delta_{(0,q_2,q_3,q_4)}=\Delta_{(0,q_2,q_4,q_3)}^{op}$, we obtain also the first assertion. \\

5. By lemma \ref{21}, point 3:
$$\Delta_{(0,0,q_3,q_4)}(P)=\sum_{\substack{I\subseteq P\\h_{P\setminus I}^I=h_I^{P\setminus I}=0}}
q_3^{r_{P\setminus I}^I}q_4^{r_I^{P\setminus I}}(R\setminus I)\otimes I
=\sum_{I\subseteq \{1,\cdots,k\}}q_3^{r_{P\setminus P_I}^{ P_I}}q_4^{r_{ P_I}^{P\setminus  P_I}}
 P_I^\otimes P_{\{1,\cdots,k\}\setminus I},$$
and it is immediate that $r_{P\setminus P_I}^{ P_I}=\alpha_P(I)$ and
$r_{ P_I}^{P\setminus  P_I}=\alpha_P(\{1,\ldots,k\}-I)$.\\

7. By lemma \ref{21}:
$$\Delta_{(0,0,q,0)}(P)=\sum_{\substack{I\subseteq P\\h_{P\setminus I}^I=h_I^{P\setminus I}=r_I^{P\setminus I}=0}}
q^{r_{P\setminus I}^I}(P\setminus I)\otimes I
= \sum_{P_1P_2=P} q^{r_{P_1}^{P_2}}P_1 \otimes P_2= \sum_{P_1P_2=P} q^{|P_1||P_2|}P_1 \otimes P_2.$$
8. Let $I\subseteq P$, such that $h_{P\setminus I}^I=h_I^{P\setminus I}=r_{P\setminus I}^I=r_I^{P\setminus I}=0$.
Then:
$$|I||P\setminus I|=h_{P\setminus I}^I+h_I^{P\setminus I}+r_{P\setminus I}^I+r_I^{P\setminus I}=0,$$
so $I=1$ or $I=P$. \end{proof}\\

{\bf Remark.} In particular, the coproduct defined in section \ref{recalls} is $\Delta_{(1,0,1,1)}$.
The coproduct of deconcatenation, dual of $m$, is $\Delta_{(0,0,1,0)}$ and the coproduct of deconcatenation, dual of $\prodh$, is $\Delta_{(1,0,0,0)}$. \\



\subsection{Compatibilities with the products}

\begin{prop}\label{23}
Let $x,y \in \h_{\PP}$. We put $\Delta_q(x)=\sum x_q'\otimes x_q''$ and $\Delta_q(y)=\sum y_q'\otimes y_q''$,
with the $x_q'$'s, $x_q''$'s, $y_q'$'s, $y_q''$'s homogeneous. Then:
\begin{enumerate}
\item $\displaystyle \Delta_q(xy)=\sum\sum q_3^{|x'_q||y''_q|}q_4^{|x''_q||y'_q|} (x'_qy'_q) \otimes (x''_qy''_q)$.
\item $\displaystyle \Delta_q(x \prodh y)=\sum\sum q_1^{|x'_q||y''_q|}q_2^{|x''_q||y'_q|} (x'_q \prodh y'_q) \otimes (x''_q \prodh y''_q)$.
\end{enumerate}
\end{prop}

\begin{proof} We only prove the first point; the proof of the second point is similar. Let $P,Q \in \PP$. Then:
\begin{eqnarray*}
\Delta_q(P Q)&=&\sum_{I\subseteq P,\: J \subseteq Q}
q_1^{h_{(P Q)\setminus (IJ)}^{IJ}} q_2^{h^{(P Q)\setminus (IJ)}_{IJ}}
q_3^{r_{(P Q)\setminus (IJ)}^{IJ}} q_4^{r^{(P Q)\setminus (IJ)}_{IJ}}
((P Q)\setminus (IJ)) \otimes (IJ)\\
&=&\sum_{I\subseteq P,\: J \subseteq Q}
q_1^{h_{(P Q)\setminus (IJ)}^{IJ}} q_2^{h^{P Q\setminus IJ}_{IJ}}
q_3^{r_{P Q\setminus IJ}^{IJ}} q_4^{r^{P Q\setminus IJ}_{IJ}}
((P\setminus I) (Q\setminus J)) \otimes( IJ).
\end{eqnarray*}
Moreover, for all $I\subseteq P$, $J \subseteq Q$:
\begin{itemize}
\item For all $x \in P\setminus I$, $y\in J$, $x <_r y$ in $PQ$, so $h_{P\setminus I}^J=0$.
Similarly, $h_{Q\setminus J}^I=0$. Hence:
$$h_{P Q\setminus IJ}^{IJ}=h_{P\setminus I}^I+h_{Q\setminus J}^J+h_{P\setminus I}^J+h_{Q\setminus J}^I 
=h_{P\setminus I}^I+h_{Q\setminus J}^J.$$
\item In the same way, $h^{P Q\setminus IJ}_{IJ}=h^{P\setminus I}_I+h^{Q\setminus J}_J+h^{P\setminus I}_J+h^{Q\setminus J}_I 
=h^{P\setminus I}_I+h^{Q\setminus J}_J$.
\item For all $x \in P\setminus I$, $y\in J$, $x <_r y$ in $PQ$, so $r_{P\setminus I}^J=|P\setminus I||J|$.
For all $x \in Q\setminus J$, $y\in I$, $x>_r y$, so $r_{Q\setminus J}^I=0$.  Hence:
$$r_{P Q\setminus IJ}^{IJ}=r_{P\setminus I}^I+r_{Q\setminus J}^J+r_{P\setminus I}^J+r_{Q\setminus J}^I 
=h_{P\setminus I}^I+h_{Q\setminus J}^J+|P\setminus I||J|.$$
\item In the same way, $r^{P Q\setminus IJ}_{IJ}=r^{P\setminus I}_I+r^{Q\setminus J}_J+r^{P\setminus I}_J+r^{Q\setminus J}_I 
=r^{P\setminus I}_I+r^{Q\setminus J}_J+|I||Q\setminus J|$.
\end{itemize}
So:
$$\Delta_q(P Q)=\sum \sum q_3^{|P_q'||Q_q''|}q_4^{|P_q''||Q_q'|} (P_q'Q_q') \otimes (P_q''Q_q''),$$
which is the announced formula. \end{proof}\\

{\bf Examples.} \begin{enumerate}
\item If $q_3=1$, then $(\h_{\PP},m,\Delta_q)$ is a braided Hopf algebra. The braiding is given, for all $P,Q \in \PP$, by:
$$c_{q_4}(P\otimes Q)=q_4^{|P||Q|} Q \otimes P.$$
In particular, if $q_4=1$, it is a Hopf algebra; if $q_4=0$, it is an infinitesimal Hopf algebra.
\item If $q_4=0$, the compatibility becomes the following: 
$$\Delta_q(xy)=\sum q_3^{|x||y''_q|} (xy'_q) \otimes y''_q +\sum q_3^{|x'_q||y|} x'_q \otimes (x''_qy)-q_3^{|x||y|} x \otimes y.$$
In particular, if $q_3=1$, then it is an infinitesimal Hopf algebra.
\item If $q_2=q_4=0$,  then for all $x,y \in \h_{\PP}$:
\begin{eqnarray*}
\Delta_q(x y)&=&(xy) \otimes 1+1\otimes (xy)+\varepsilon(x)(\Delta_q(y)-y\otimes 1-1\otimes y)\\
&&+\varepsilon(y)(\Delta_q(x)-x \otimes 1-1 \otimes x)+\varepsilon(x)\varepsilon(y)1\otimes 1.
\end{eqnarray*}
In other terms, for all $x,y \in \h_{\PP}$, such that $\varepsilon(x)=\varepsilon(y)=0$, $x y$ is primitive.
\end{enumerate}

\section{Self-duality results}

\subsection{A first pairing on $\h_q$}

{\bf Notations.} If $P,Q \in \PP$, we denote by $Bij(P,Q)$ the set of bijections from $P$ to $Q$.

\begin{defi}
Let $P,Q$ be two double posets and let $\sigma\in Bij(P,Q)$. We put:
$$\left\{\begin{array}{rcl}
\phi_1(\sigma)&=&\sharp\{(x,y)\in P^2\mid x<_h y \mbox{ and  }\sigma(x)<_h \sigma(y)\}\\
&&+\sharp\{(x,y)\in P^2\mid x<_h y \mbox{ and  }\sigma(x)>_h \sigma(y)\}\\
&&+\sharp\{(x,y)\in P^2\mid x<_h y \mbox{ and  }\sigma(x)<_r \sigma(y)\}\\
&&+\sharp\{(x,y)\in P^2\mid x<_r y \mbox{ and  }\sigma(x)<_h \sigma(y)\},\\[2mm]
\phi_2(\sigma)&=&\sharp\{(x,y)\in P^2\mid x<_h y \mbox{ and  }\sigma(x)<_h \sigma(y)\}\\
&&+\sharp\{(x,y)\in P^2\mid x<_h y \mbox{ and  }\sigma(x)>_h \sigma(y)\}\\
&&+\sharp\{(x,y)\in P^2\mid x<_h y \mbox{ and  }\sigma(x)>_r \sigma(y)\}\\
&&+\sharp\{(x,y)\in P^2\mid x<_r y \mbox{ and  }\sigma(x)>_h \sigma(y)\},\\[2mm]
\phi_3(\sigma)&=&\sharp\{(x,y) \in P^2\mid x<_r y \mbox{ and }\sigma(x)<_r \sigma(y)\},\\[2mm]
\phi_4(\sigma)&=&\sharp\{(x,y) \in P^2\mid x<_r y \mbox{ and }\sigma(x)>_r \sigma(y)\}.
\end{array}\right.$$
\end{defi}

\begin{lemma}\label{25}
Let $P,Q$ be two double posets and let $\sigma\in Bij(P,Q)$. For all $i\in \{1,\ldots,4\}$, $\phi_i(\sigma^{-1})=\phi_i(\sigma)$.
\end{lemma}

\begin{proof} Indeed, as $\sigma$ is a bijection:
\begin{eqnarray*}
\phi_1(\sigma^{-1})&=&\sharp\{(x,y)\in Q^2\mid x<_h y \mbox{ and  }\sigma^{-1}(x)<_h \sigma^{-1}(y)\}\\
&&+\sharp\{(x,y)\in P^2\mid x<_h y \mbox{ and  }\sigma^{-1}(x)>_h \sigma^{-1}(y)\}\\
&&+\sharp\{(x,y)\in P^2\mid x<_h y \mbox{ and  }\sigma^{-1}(x)<_r \sigma^{-1}(y)\}\\
&&+\sharp\{(x,y)\in P^2\mid x<_r y \mbox{ and  }\sigma^{-1}(x)<_h \sigma^{-1}(y)\}\\
&=&\sharp\{(x,y)\in P^2\mid x<_h y \mbox{ and  }\sigma(x)<_h \sigma(y)\}\\
&&+\sharp\{(x,y)\in P^2\mid x>_h y \mbox{ and  }\sigma(x)<_h \sigma(y)\}\\
&&+\sharp\{(x,y)\in P^2\mid x<_h y \mbox{ and  }\sigma(x)<_r \sigma(y)\}\\
&&+\sharp\{(x,y)\in P^2\mid x<_r y \mbox{ and  }\sigma(x)<_h \sigma(y)\}\\
&=&\phi_1(\sigma).
\end{eqnarray*}
The other equalities are proved similarly. \end{proof}

\begin{lemma}\label{26}
Let $P_1,P_2,Q$ be double posets. There is a bijection:
$$\left\{\begin{array}{rcl}
Bij(P_1P_2,Q)&\longrightarrow&\displaystyle \bigcup_{I\subseteq Q}Bij(P_1,R\setminus I)\times R(P_2,I)\\
\sigma&\longrightarrow&(\sigma_{\mid P_1},\sigma_{\mid P_2}), \mbox{ with }I=\sigma(P_2).
\end{array}\right.$$
Let $\sigma \in Bij(P_1P_2,Q)$ and let $(\sigma_1,\sigma_2)$ be its image by this bijection. Then:
$$\left\{\begin{array}{rcl}
\phi_1(\sigma)&=&\phi_1(\sigma_1)+\phi_1(\sigma_2)+h_{Q\setminus I}^I,\\[2mm]
\phi_2(\sigma)&=&\phi_2(\sigma_1)+\phi_2(\sigma_2)+h_I^{Q\setminus I},\\[2mm]
\phi_3(\sigma)&=&\phi_3(\sigma_1)+\phi_3(\sigma_2)+r_{Q\setminus I}^I,\\[2mm]
\phi_4(\sigma)&=&\phi_4(\sigma_1)+\phi_4(\sigma_2)+r_I^{Q\setminus I}.
\end{array}\right.$$
\end{lemma}

\begin{proof} We put $P=P_1P_2$. if $x \in P_1$ and $y\in P_2$, then $x<_r y$, so:
\begin{eqnarray*}
\phi_3(\sigma)&=&\sharp\{(x,y) \in P_1^2\mid x<_r y \mbox{ and }\sigma(x)<_r \sigma(y)\}\\
&&+\sharp\{(x,y) \in P_2^2\mid x<_r y \mbox{ and }\sigma(x)<_r \sigma(y)\}\\
&&+\sharp\{(x,y) \in P_1 \times P_2\mid x<_r y \mbox{ and }\sigma(x)<_r \sigma(y)\}\\
&&+\sharp\{(x,y) \in P_2 \times P_1\mid x<_r y \mbox{ and }\sigma(x)<_r \sigma(y)\}\\
&=&\phi_3(\sigma)+\phi_3(\sigma_2)+\sharp\{(x,y) \in P_1 \times P_2\mid \sigma(x)<_r \sigma(y)\}+0\\
&=&\phi_3(\sigma)+\phi_3(\sigma_2)+\sharp\{(x,y) \in (Q\setminus I)\times I\mid x <_r y\}\\
&=&\phi_3(\sigma_1)+\phi_3(\sigma_2)+r_{Q\setminus I}^I.
\end{eqnarray*}
The other equalities are proved similarly. \end{proof}

\begin{theo} 
Let $q=(q_1,q_2,q_3,q_4)\in K^4$. We define a pairing on $\h_{\PP}$ by:
$$\langle P,Q\rangle_q=\sum_{\sigma \in Bij(P,Q)} q_1^{\phi_1(\sigma)}
 q_2^{\phi_2(\sigma)} q_3^{\phi_3(\sigma)}q_3^{\phi_4(\sigma)}.$$
 This pairing is symmetric and for all $x,y,z \in \h_{\PP}$, $\langle xy,z\rangle_q=\langle x \otimes y,\Delta_q(z)\rangle_q$.
\end{theo}

\begin{proof}
Let $P,Q$ be two double posets. Then, by lemma \ref{25}:
\begin{eqnarray*}
\langle P,Q\rangle_q&=&\sum_{\sigma \in Bij(P,Q)} q_1^{\phi_1(\sigma)}
 q_2^{\phi_2(\sigma)} q_3^{\phi_3(\sigma)}q_3^{\phi_4(\sigma)}\\
 &=&\sum_{\sigma \in Bij(Q,P)}q_1^{\phi_1(\sigma^{-1})} q_2^{\phi_2(\sigma^{-1})}
q_3^{\phi_3(\sigma^{-1})}q_3^{\phi_4(\sigma^{-1})}\\
 &=&\sum_{\sigma \in Bij(Q,P)} q_1^{\phi_1(\sigma)} q_2^{+\phi_2(\sigma)} q_3^{\phi_3(\sigma)}q_3^{\phi_4(\sigma)}\\
 &=&\langle Q,P\rangle_q.
\end{eqnarray*}
So this pairing is symmetric. Let $P_1,P_2,Q$ be three double posets. By lemma \ref{26}:
\begin{eqnarray*}
\langle P_1P_2,Q\rangle_q&=&\sum_{\sigma \in Bij(P_1P_2,Q)} q_1^{\phi_1(\sigma)}
 q_2^{\phi_2(\sigma)} q_3^{\phi_3(\sigma)}q_3^{\phi_4(\sigma)}\\
 &=&\sum_{I\subseteq Q}\sum_{\sigma_1 \in Bij(P_1,Q\setminus I)}\sum_{\sigma_2\in Bij(P_2,I)}
 q_1^{\phi_1(\sigma_1)+\phi_1(\sigma_2)+h_{Q\setminus I}^I}\\
 &&\times q_2^{\phi_2(\sigma_1)+\phi_2(\sigma_2)+h^{Q\setminus I}_I}
 q_3^{\phi_3(\sigma_1)+\phi_3(\sigma_2)+r_{Q\setminus I}^I}
 q_4^{\phi_4(\sigma_1)+\phi_4(\sigma_2)+r^{Q\setminus I}_I}\\
 &=&\sum_{I\subseteq Q}q_1^{h_{R\setminus I}^I}q_2^{h_I^{R\setminus I}}q_3^{r_{R\setminus I}^I}q_4^{r_I^{R\setminus I}}
\langle P_1,Q\setminus I \rangle_q\langle P_2,I\rangle_q\\
&=&\langle P_1\otimes P_2,\Delta_q(Q)\rangle_q.
\end{eqnarray*}
So $\langle-,-\rangle_q$ is a Hopf pairing. \end{proof} \\

{\bf Remark.} More generally, adapting this proof, it is possible to show that for all $x,y,z \in \h_{\PP}$:
$$\langle m_{(0,0,q,0)}(x\otimes y),\Delta(z)\rangle_{(q_1,q_2,q_3,q_4)}
=\langle x \otimes y,\Delta_{(qq_1,qq_2,qq_3,qq_4)}(z)\rangle_{(q_1,q_2,q_3,q_4)}.$$

{\bf Examples.} $\langle \tun, \tun \rangle_q=1$. The pairing of plane posets of degree $2$  is given by the following array:
$$\begin{array}{c|c|c}
&\tdeux&\tun\tun\\
\hline \tdeux&2q_1q_2&q_1+q_2\\
\hline \tun\tun&q_1+q_2&q_3+q_4
\end{array}$$

\begin{prop}
For all double posets $P,Q$:
$$\langle P,Q \rangle_{(q_1,0,q_3,q_4)}
=\sum_{\sigma \in S(P,Q)} q_1^{\phi_1(\sigma)}q_3^{\phi_3(\sigma)}q_4^{\phi_4(\sigma)}.$$
Consequently, $\langle-,-\rangle_{(1,0,1,1)}$ is the pairing $\langle-,-\rangle$ described in section \ref{recalls}.
\end{prop}

\begin{proof} Let $\sigma \in Bij(P,Q)$, such that the contribution of $\sigma$ in the sum defining $\langle P,Q \rangle_{(q_1,0,q_3,q_4)}$
is non zero. As $q_2=0$, necessarily $\phi_2(\sigma)=0$. Hence, if $x<_h y$ in $P$, then $\sigma(x)<_r \sigma(y)$ in $Q$;
if $\sigma(x)<_h \sigma(y)$ in $Q$, then $x<_r y$ in $P$. So $\sigma \in S(P,Q)$. \end{proof}\\

\subsection{Properties of the pairing}

\begin{lemma}\label{29}
Let $P_1,P_2,Q$ be double posets. There is a bijection:
$$\left\{\begin{array}{rcl}
Bij(P_1\prodh P_2,Q)&\longrightarrow&\displaystyle \bigcup_{I\subseteq Q}Bij(P_1,R\setminus I)\times R(P_2,I)\\
\sigma&\longrightarrow&(\sigma_{\mid P_1},\sigma_{\mid P_2}), \mbox{ with }I=\sigma(P_2).
\end{array}\right.$$
Let $\sigma \in Bij(P_1P_2,Q)$ and let $(\sigma_1,\sigma_2)$ be its image by this bijection. Then:
$$\left\{\begin{array}{rcl}
\phi_1(\sigma)&=&\phi_1(\sigma_1)+\phi_1(\sigma_2)+h_{Q\setminus I}^I+h_I^{Q\setminus I}+r_{Q\setminus I}^I,\\[2mm]
\phi_2(\sigma)&=&\phi_2(\sigma_1)+\phi_2(\sigma_2)+h_{Q\setminus I}^I+h_I^{Q\setminus I}+r^{Q\setminus I}_I,\\[2mm]
\psi_2(\sigma)&=&\psi_2(\sigma_1)+\psi_2(\sigma_2)+h_{Q\setminus I}^I+h_I^{Q\setminus I},\\[2mm]
\phi_3(\sigma)&=&\phi_3(\sigma_1)+\phi_3(\sigma_2),\\[2mm]
\phi_4(\sigma)&=&\phi_4(\sigma_1)+\phi_4(\sigma_2).
\end{array}\right.$$
\end{lemma}

\begin{proof} We put $P=P_1 \prodh P_2$. Let $x,y \in P$. If $x<_h y,$ then $x,y\in P_1$, or $x,y\in P_2$,  or $x\in P_1,y\in P_2$.
If $x<_r y$, then $x,y\in P_1$, or $x,y\in P_2$. Moreover, if $x\in P_1$ and $x \in P_2$, then $x<_h y$. So:
\begin{eqnarray*}
\phi_1(\sigma)&=&\phi_1(\sigma_1)+\phi_1(\sigma_2)\\
&&+\sharp\{(x,y)\in P_1 \times P_2\mid (x<_h y) \mbox{ and  }\sigma(x)<_h \sigma(y))\}\\
&&+\sharp\{(x,y)\in  P_1 \times P_2\mid (x<_h y) \mbox{ and  }\sigma(x)>_h \sigma(y))\}\\
&&+\sharp\{(x,y)\in  P_1 \times P_2\mid (x<_h y) \mbox{ and  }\sigma(x)<_r \sigma(y))\}\\[2mm]
&=&\phi_1(\sigma_1)+\phi_1(\sigma_2)\\
&&+\sharp\{(x,y)\in P_1 \times P_2\mid \sigma(x)<_h \sigma(y))\}\\
&&+\sharp\{(x,y)\in  P_1 \times P_2\mid \sigma(x)>_h \sigma(y))\}\\
&&+\sharp\{(x,y)\in  P_1 \times P_2\mid \sigma(x)<_r \sigma(y))\}\\[2mm]
&=&\phi_1(\sigma_1)+\phi_1(\sigma_2)+h_{Q\setminus I}^I+h_I^{Q\setminus I}+r_{Q\setminus I}^I.
\end{eqnarray*}
The other equalities are proved similarly. \end{proof}

\begin{prop}
For all $x,y,z \in \h_{\PP}$, $\langle x\prodh y,z\rangle_q=\langle x\otimes y, \Delta_{(q_1q_2,q_1q_2,q_1,q_2)}(z)\rangle_q$.
\end{prop}

\begin{proof} Let $P_1,P_2,Q$ be three double posets. By lemma \ref{29}:
\begin{eqnarray*}
\langle P_1\prodh P_2,Q\rangle_q&=&\sum_{\sigma \in Bij(P_1\prodh P_2,Q)} q_1^{\phi_1(\sigma)}
 q_2^{\phi_2(\sigma)} q_3^{\phi_3(\sigma)}q_3^{\phi_4(\sigma)}\\
 &=&\sum_{I\subseteq Q}\sum_{\sigma_1 \in Bij(P_1,Q\setminus I)}\sum_{\sigma_2\in Bij(P_2,I)}
 q_1^{\phi_1(\sigma_1)+\phi_1(\sigma_2)+h_{Q\setminus I}^I+h_I^{Q\setminus I}+r_{Q\setminus I}^I}\\
 &&\times q_2^{\phi_2(\sigma_1)+\phi_2(\sigma_2)+h_{Q\setminus I}^I+h_I^{Q\setminus I}+r^{Q\setminus I}_I}
 q_3^{\phi_3(\sigma_1)+\phi_3(\sigma_2)}
 q_4^{\phi_4(\sigma_1)+\phi_4(\sigma_2)}\\
 &=&\sum_{I\subseteq Q}(q_1q_2)^{h_{R\setminus I}^I}(q_1q_2)^{h_I^{R\setminus I}}
q_1^{r_{R\setminus I}^I}q_2^{r_I^{R\setminus I}}\langle P_1,Q\setminus I \rangle_q\langle P_2,I\rangle_q\\
&=&\langle P_1\otimes P_2,\Delta_{(q_1q_2,q_1q_2,q_1,q_2)}(Q)\rangle_q,
\end{eqnarray*} 
which is the announced formula. \end{proof}\\

{\bf Remarks.} \begin{enumerate}
\item In particular, if $q=(1,0,1,1)$, for all $P,Q,R \in \PP$:
$$\langle P\prodh Q,R\rangle=\langle P\otimes Q,\Delta_{(0,0,1,0)}(R)\rangle=\sum_{R=R_1R_2} \langle P\otimes Q,R_1 \otimes R_2\rangle.$$
this formula is already proved in \cite{F1}.
\item  More generally, it is possible to show that for all $x,y,z \in \h_{\PP}$:
$$\langle m_{(q,0,0,0)}(x\otimes y),\Delta(z)\rangle_{(q_1,q_2,q_3,q_4)}
=\langle x \otimes y,\Delta_{(qq_1q_2,qq_1q_2,qq_1,qq_2)}(z)\rangle_{(q_1,q_2,q_3,q_4)}.$$
\end{enumerate}

\begin{prop}\label{31}
For all $x,y \in \h{\PP}$:
$$\begin{array}{rclccrcl}
\langle x,\beta(y)\rangle_{(q_1,q_2,q_3,q_4)}&=&\langle x,y\rangle_{(q_1,q_2,q_4,q_3)},&&
\langle x,\gamma(y)\rangle_{(q_1,q_2,q_3,q_4)}&=&\langle x,y\rangle_{(q_2,q_1,q_4,q_3)},\\
\langle \alpha(x),\alpha(y)\rangle_{(q_1,q_2,q_3,q_4)}&=&\langle x,y \rangle_{(q_2,q_1,q_3,q_4)},&&
\langle \alpha(x),\beta(y)\rangle_{(q_1,q_2,q_3,q_4)}&=&\langle x,y \rangle_{(q_1,q_2,q_4,q_3)},\\
\langle \alpha(x),\gamma(y)\rangle_{(q_1,q_2,q_3,q_4)}&=&\langle x,y \rangle_{(q_2,q_1,q_4,q_3)},&&
\langle \beta(x),\beta(y)\rangle_{(q_1,q_2,q_3,q_4)}&=&\langle x,y \rangle_{(q_2,q_1,q_3,q_4)},\\
\langle \gamma(x),\gamma(y)\rangle_{(q_1,q_2,q_3,q_4)}&=&\langle x,y \rangle_{(q_1,q_2,q_3,q_4)}.
\end{array}$$
\end{prop}

\begin{proof} Let $P,Q$ be two double posets. We put $P'=\alpha(P)$ and $Q'=\alpha(Q)$. Note that $Bij(P,Q)=Bij(P',Q')$.
Let $\sigma\in Bij(P,Q)$. We denote it by $\sigma'$ is we consider it as an element of $Bij(P',Q')$. By definition of $P'$ and $Q'$, it is clear that
$\phi_1(\sigma')=\phi_2(\sigma)$, $\phi_2(\sigma')=\phi_2(\sigma)$, $\phi_3(\sigma')=\phi_3(\sigma)$, and $\phi_4(\sigma')=\phi_4(\sigma)$.
Hence:
\begin{eqnarray*}
\langle P',Q' \rangle_{q_1,q_2,q_3,q_4)}&=&\sum_{\sigma'\in Bij(P',Q')}q_1^{\phi_1(\sigma')}
q_2^{\phi_2(\sigma')}q_3^{\phi_3(\sigma')}q_4^{\phi_4(\sigma')}\\
&=&\sum_{\sigma\in Bij(P,Q)}q_1^{\phi_2(\sigma)}q_2^{\phi_1(\sigma)}q_3^{\phi_3(\sigma)}q_4^{\phi_4(\sigma)}\\
&=&\langle P,Q \rangle_{(q_2,q_1,q_3,q_4)}.
\end{eqnarray*}
The other assertions are proved in the same way. \end{proof}\\

{\bf Remark.}  In general, the pairing defined by $x\otimes y\longrightarrow \langle x,\alpha(y)\rangle_{(q_1,q_2,q_3,q_4)}$ is not a pairing
$\langle-,-\rangle_q$. However, when $q_1=q_2$, it is possible to prove that for all $x,y \in \h_{\PP}$, 
$\langle x,\alpha(y)\rangle_{(q_1,q_1,q_3,q_4)}=\langle x,y\rangle_{(q_1,q_1,q_3,q_4)}$. Similarly, 
$\langle \beta(x),\gamma(y)\rangle_{(q_1,q_1,q_3,q_4)}=\langle x,y\rangle_{(q_1,q_1,q_4,q_3)}$.

\subsection{Comparison of pairings with colinear parameters}

\begin{prop}
Let $q \in K$. We define the following map:
$$\upsilon_q:\left\{\begin{array}{rcl}
\h_{\PP}&\longrightarrow&\h_{\PP}\\
P\in \PP&\longrightarrow&q^{h(P)}P,
\end{array}\right.$$
where $h(P)=\sharp\{(x,y)\in P\mid x<_h y\}$. Then $\upsilon_q$ is an algebra and coalgebra morphism from 
$\left(\h_{\PP},m,\Delta_{(qq_1,qq_2,q_3,q_4)}\right)$ to $\left(\h_{\PP},m,\Delta_{(q_1,q_2,q_3,q_4)}\right)$.
Moreover, for all $x,y \in \h_{\PP}$:
$$\langle \upsilon_q(x),\upsilon_q(y)\rangle_{(q_1,q_2,q_3,q_4)}=\langle x,y\rangle_{(qq_1,qq_2,q_3,q_4)}.$$
\end{prop}

\begin{proof} Let $P_1,P_2 \in \PP$. Then $h(P_1P_2)=h(P_1)+h(P_2)$, so $\upsilon_q(P_1P_2)=\upsilon_q(P_1)\upsilon_q(P_2)$, 
and $\upsilon_q$ is an algebra morphism. Let $P \in \PP$. For any $I\subseteq P$, as $P^2=I^2\sqcup (P\setminus I)^2\sqcup (I \times (P\setminus I))
\sqcup ((P\setminus I)\times I)$:
$$h(P)=h(I)+h(P\setminus I)+h_{P\setminus I}^I+h_I^{P\setminus I}.$$
Consequently:
\begin{eqnarray*}
\Delta_{(q_1,q_2,q_3,q_4)}\circ \upsilon_q(P)&=&\sum_{I\subseteq P}(qq_1)^{h_{P\setminus I}^I}
(qq_2)^{h_I^{P\setminus I}}q_3^{r_{P\setminus I}^I}q_4^{r_I^{P\setminus I}}q^{h(P\setminus I)}(P\setminus I)\otimes q^{h(I)}I\\
&=&(\upsilon_q \otimes \upsilon_q)\circ \Delta_{(qq_1,qq_2,q_3,q_4)}(P).
\end{eqnarray*}
So $\upsilon_q$ is a coalgebra morphism. Let $P,Q \in \PP$ and let $\sigma \in Bij(P,Q)$. We define:
\begin{eqnarray*}
a_1&=&\sharp\{(x,y)\in P^2\mid x<_hy,\sigma(x)<_h \sigma(y)\}\\
a_2&=&\sharp\{(x,y)\in P^2\mid x<_hy,\sigma(x)>_h \sigma(y)\}\\
a_3&=&\sharp\{(x,y)\in P^2\mid x<_hy,\sigma(x)<_r \sigma(y)\}\\
a_4&=&\sharp\{(x,y)\in P^2\mid x<_hy,\sigma(x)>_r \sigma(y)\}\\
a_5&=&\sharp\{(x,y)\in P^2\mid x<_ry,\sigma(x)<_h \sigma(y)\}\\
a_6&=&\sharp\{(x,y)\in P^2\mid x<_ry,\sigma(x)>_h \sigma(y)\}\\
a_7&=&\sharp\{(x,y)\in P^2\mid x<_ry,\sigma(x)<_r \sigma(y)\}\\
a_8&=&\sharp\{(x,y)\in P^2\mid x<_ry,\sigma(x)>_r \sigma(y)\}.
\end{eqnarray*}
In order to sum up the notations, we give a following array:
$$\begin{array}{c|c|c|c|c}
\sharp\{(x,y)\in P^2\mid \ldots&x<_hy,\ldots&x>_hy,\ldots&x<_ry,\ldots&x>_ry,\ldots\\
\hline \sigma(x)<_h \sigma(y)\}&a_1&a_2&a_5&a_6\\
\hline \sigma(x)>_h \sigma(y)\}&a_2&a_1&a_6&a_5\\
\hline \sigma(x)<_r \sigma(y)\}&a_3&a_4&a_7&a_8\\
\hline \sigma(x)>_r \sigma(y)\}&a_4&a_3&a_8&a_7
\end{array}$$
So $\phi_1(\sigma)=a_1+a_2+a_3+a_5$, $\phi_2(\sigma)=a_1+a_2+a_4+a_6$, $\phi_3(\sigma)=a_7$ and $\phi_4(\sigma)=a_8$;
$h(P)=a_1+a_2+a_3+a_4$ and $h(Q)=a_1+a_2+a_5+a_6$. Then:
\begin{eqnarray*}
q^{h(P)}q^{(h(Q)}q_1^{\phi_1(\sigma)}q_2^{\phi_2(\sigma)}q_3^{\phi_3(\sigma)}q_4^{\phi_4(\sigma)}
&=&(q^2q_1q_2)^{a_1}(q^2q_1q_2)^{a_2}(qq_1)^{a_3}(qq_2)^{a_4}(qq_1)^{a_5}(qq_2)^{a_6}q_3^{a_7}q_4^{a_8}\\
&=&(qq_1)^{\phi_1(\sigma)}(qq_2)^{\phi_2(\sigma)}q_3^{\phi_3(\sigma)}q_4^{\phi_4(\sigma)}.
\end{eqnarray*}
Finally:
\begin{eqnarray*}
\langle \upsilon_q(P),\upsilon_q(Q)\rangle_{(q_1,q_2,q_3,q_4)}&=&\sum_{\sigma \in Bij(P,Q)}
q^{h(P)}q^{(h(Q)}q_1^{\phi_1(\sigma)}q_2^{\phi_2(\sigma)}q_3^{\phi_3(\sigma)}q_4^{\phi_4(\sigma)}\\
&=&\sum_{\sigma \in Bij(P,Q)}(qq_1)^{\phi_1(\sigma)}(qq_2)^{\phi_2(\sigma)}q_3^{\phi_3(\sigma)}q_4^{\phi_4(\sigma)}\\
&=&\langle P,Q \rangle_{(qq_1,qq_2,q_3,q_4)}.
\end{eqnarray*}
So $\upsilon_q$ is an isometry. \end{proof}\\

{\bf Remark.} It is easy to prove that for $P,Q \in \PP$, $\upsilon_q(P\prodh Q)=\upsilon_q(P)\prodh \upsilon_q(Q)$.\\

Similarly, one can prove:
\begin{prop}
Let $q \in K$. We define the following map:
$$\upsilon'_q:\left\{\begin{array}{rcl}
\h_{\PP}&\longrightarrow&\h_{\PP}\\
P\in \PP&\longrightarrow&q^{r(P)}P,
\end{array}\right.$$
where $r(P)=\sharp\{(x,y)\in P\mid x<_r y\}$. Then $\upsilon'_q$ is an algebra and coalgebra morphism from
$\left(\h_{\PP},\prodh,\Delta_{(q_1,q_2,qq_3,qq_4)}\right)$ to $\left(\h_{\PP},\prodh,\Delta_{(q_1,q_2,q_3,q_4)}\right)$.
Moreover, for all $x,y \in \h_{\PP}$:
$$\langle \upsilon'_q(x),\upsilon'_q(y)\rangle_{(q_1,q_2,q_3,q_4)}=\langle x,y\rangle_{(q_1,q_2,qq_3,qq_4)}.$$
\end{prop}

\subsection{Non-degeneracy of the pairing}

\begin{lemma} \label{34}
For all $P\in \PP(n)$, $S(P,\iota(P))$ is reduced to a single element and:
$$\langle P,\iota(P)\rangle_{(q_1,0,q_3,q_4)}=q_1^{\frac{n(n-1)}{2}}.$$
\end{lemma}

\begin{proof} Clearly, $Id_P:P\longrightarrow P$ belongs to $S(P,\iota(P))$.
Let $\sigma:P \longrightarrow P$ be a bijection. Then $\sigma \in S(P,\iota(P))$ if, and only if, for all $i,j \in P$:
\begin{itemize}
\item ($i\leq_h j$ in $P$) $\Longrightarrow$ ($\sigma(i)\leq_h \sigma(j)$ in $P$).
\item ($\sigma(i)\leq_r \sigma(j)$) in $P$ $\Longrightarrow$ ($i\leq_r j$ in $P$).
\end{itemize}
It is clear $S(P,\iota(P))$ is a submonoid of $\S_P$. As the group $\S_P$ is finite, $S(P,\iota(P))$ is a subgroup of $\S_P$. So if $\sigma \in S(P,\iota(P))$, 
then $\sigma^{-1} \in S(P,\iota(P))$. So, if $\sigma \in S(P,\iota(P))$:
 \begin{itemize}
\item ($i\leq_h j$ in $P$) $\Longleftrightarrow$ ($\sigma(i)\leq_h \sigma(j)$ in $P$).
\item ($i\leq_r j$ in $P$) $\Longleftrightarrow$ ($\sigma(i)\leq_r \sigma(j)$ in $P$).
\end{itemize}
So $\sigma$ is the unique increasing bijection from $P$ to $P$, that is to say $Id_P$. 
Moreover:
\begin{eqnarray*}
\phi_1(Id_P)&=&\sharp\{(x,y) \in P^2\mid x<_h y \mbox{ and }x<_ry\}+\sharp\{(x,y) \in P^2\mid x<_h y \mbox{ and }x>_ry\}\\
&&+\sharp\{(x,y) \in P^2\mid x<_h y \mbox{ and }x<_hy\}+\sharp\{(x,y) \in P^2\mid x<_r y \mbox{ and }x<_ry\}\\
&=&0+0+\sharp\{(x,y)\in P^2\mid x<_h y\}+\sharp\{(x,y)\in P^2\mid x<_r y\}\\
&=&\sharp\{(x,y)\in P^2\mid x< y\}\\
&=&\frac{n(n-1)}{2},\\ 
\phi_3(Id_P)&=&\sharp\{(x,y) \in P^2\mid (x<_r y)\mbox{ and }(x<_h y)\}\\
&=&0,\\
\phi_4(Id_P)&=&\sharp\{(x,y) \in P^2\mid (x>_r y)\mbox{ and }(x>_h y)\}\\
&=&0.
\end{eqnarray*}
So $\langle P,\iota(P) \rangle_{(q_1,0,q_2,q_3)}=q_1^{\frac{n(n-1)}{2}}$. \end{proof}\\

For all $n \in \mathbb{N}$, we give $\S_n$ the lexicographic order. For example, if $n=3$:
$$(123)\leq (132)\leq (213)\leq (231)\leq (312)\leq(321).$$
We then define a total order $\ll$ on $\PP(n)$ by $P\ll Q$ if, and only if, $\Psi_n(P)\leq \Psi_n(Q)$ in $\S_n$. For example, if $n=3$:
$$\ttroisdeux \ll \ttroisun \ll \ptroisun \ll \tdeux \tun \ll \tun \tdeux \ll \tun\tun\tun.$$
For all $P\in \PP(n)$, we put: 
$$m(P)=\min_\ll\{Q \in \PP(n)\:/\: S(P,Q)\neq \emptyset\}.$$

\begin{lemma} \label{35}
For all $P\in \PP$, $m(P)=\iota(P)$.
\end{lemma}

\begin{proof}  By lemma \ref{34}, $m(P)\ll \iota(P)$. Let $Q \in \PP$, such that $S(P,Q) \neq \emptyset$.
Let us prove that $\iota(P) \ll Q$. We denote $\sigma=\Psi_n(P)$ and $\tau=\Psi_n(Q)$; we can suppose that $P=\Phi_n(\sigma)$ and $Q=\Phi_n(\tau)$.
Moreover, it is not difficult to prove that $\Psi_n(\iota(P))=(n\cdots 1)\circ \sigma$.\\

{\it First step.} Let us prove that $\tau(1) \geq n-\sigma(1)+1$. In $P=\Psi_n(\sigma)$, $1\leq_h j$ if, and only if, $\sigma(1)\leq \sigma(j)$.
So there are exactly  $n-\sigma(1)+1$ elements of $P$ which satisfy $1\leq_h j$ in $P$. Let $\alpha \in S(P,Q)$ (which is non-empty by hypothesis). 
Then if $1\leq_h j$, $\alpha(1)\leq_r \alpha(j)$: there are at least $n-\sigma(1)+1$ elements of $Q$ which satisfy $\alpha(1)\leq_r j$.
Let us put $\alpha^{-1}(1)=i$. Then $\alpha(i)=1\leq \alpha(1)$ in $Q$, that is to say $1\leq_h \alpha(1)$ or $1\leq_r \alpha(1)$ in $Q$. 
If $1=\alpha(i)\leq_h \alpha(1)$ in $Q$, then $i\leq_r 1$ in $P$, so $i=1$: in both cases, $1\leq_r \alpha(1)$ in $P$.
So there are at least $n-\sigma(1)+1$ elements of $Q$ which satisfy $1\leq_r j$. In $Q$, $1\leq_r j$ if, and only if, $\tau(j)\leq \tau(1)$, 
so there are exactly $\tau(1)$ such elements. As a consequence, $\tau(1)\geq n-\sigma(1)+1$. Moreover, if there is equality, necessarily
$1=\alpha(1)$ for all $\alpha \in S(P,Q)$.\\

{\it Second step.} We consider the assertion $H_i$: if $\tau(1)=n-\sigma(1)+1$, $\cdots$, $\tau(i-1)=n-\sigma(i-1)+1$, then $\tau(i)\geq n-\sigma(i)+1$ 
and, if there is equality, then $\alpha(1)=1$, $\cdots$, $\alpha(i)=i$ for all $\alpha \in S(P,Q)$. We proved $H_0$ in the first step. 
Let us prove $H_i$ by induction on $i$. Let us assume $H_{i-1}$, $1\leq i \leq n$, and let us prove $H_i$. Let $\alpha\in S(P,Q)$ (non empty by hypothesis).
By $H_{i-1}$, as the equality is satisfied, $\alpha(1)=1$, $\cdots$, $\alpha(i-1)=i-1$.

In $P$, the number $N_i$ of elements $j$ such that $i\leq_h j$ is the cardinality of $\sigma^{-1}(\{\sigma(i),\cdots,n\}) \cap \{i,\cdots,n\}$,
so $N_i=n-\sigma(i)+1-|\{k\leq i\:/\: \sigma(k)>\sigma(i)\}|$.
Using $\alpha$, there exists at least $N_i$ elements $j$ of $Q$ such that $\alpha(i)\leq_r j$ in $Q$.

Let us put $\alpha^{-1}(i)=j$. As $\alpha(1)=1$, $\cdots$, $\alpha(i-1)=i-1$, $j\geq i$ and $i\leq \alpha(i)$. If $i\leq_h \alpha(i)$ in $Q$, 
then $j \leq_r i$ in $P$, so $j=i$. So we always have $i\leq_r \alpha(i)$ in $Q$, so at least $N_i$ elements $k$ of $Q$ satisfy $i\leq_r j$ in $Q$. Hence:
\begin{eqnarray*}
N_i&\leq&|\tau^{-1}(\{1,\cdots, \tau(i)\}) \cap \{i,\cdots,n\}|\\
n-\sigma(i)+1-|\{k\leq i\:/\: \sigma(k)>\sigma(i)\}|&\leq &\tau(i)-|\{k\leq i\:/\: \tau(k)<\tau(i)\}|\\
&\leq &\tau(i)-|\{k\leq i\:/\: n-\sigma(k)+1<n-\sigma(i)+1\}|\\
&\leq &\tau(i)-|\{k\leq i\:/\: \sigma(k)>\sigma(i)\}|,
\end{eqnarray*}
so $\tau(i) \geq n-\sigma(i)+1$. If there is equality, then necessarily $i=\alpha(i)$ for all $\alpha \in S(P,Q)$.\\

{\it Conclusion.} The hypothesis $H_i$ is true for all $0\leq i \leq n$. So $\tau\geq (n \cdots 1)\circ \sigma$ in $\S_n$,
so $\iota(P) \ll Q$ in $\PP(n)$. As a conclusion, $\iota(P) \ll m(P)$. \end{proof}

\begin{cor} \label{36}
Let us assume that $q_2=0$. Then the pairing $\langle-,-\rangle_q$ is non-degenerate if, and only if, $q_1 \neq 0$.
\end{cor}

\begin{proof} Let us fix an integer $n \in \mathbb{N}$. We consider the basis $\PP(n)$ of $\h(n)$, totally ordered by $\ll$. In this basis, 
the matrix of ${\langle-,-\rangle_q}$ restricted to $\h(n)$ has the following form, coming from lemmas \ref{34} and \ref{35}:
$$\left(\begin{array}{ccccc}
0&\cdots&\cdots&0&q_1^{\frac{n(n-1)}{2}}\\
\vdots&&\addots&q_1^{\frac{n(n-1)}{2}}&*\\
\vdots&\addots&\addots&\addots&\vdots\\
0&q_1^{\frac{n(n-1)}{2}}&\addots&&\vdots\\
q_1^{\frac{n(n-1)}{2}}&*&\cdots&\cdots&*
\end{array}\right),$$
so its determinant is $\pm q_1^{\frac{n!n(n-1)}{2}}$. Hence, $\langle-,-\rangle_q$ is non-degenerate
if, and only if, $q_1 \neq 0$. \end{proof}\\

{\bf Remark.} With the help of the isometry $\alpha$ (proposition \ref{31}), it is possible to prove that if $q_1=0$,
$\langle-,-\rangle_q$ is non-degenerate if, and only if, $q_2 \neq 0$.

\begin{cor} 
If $q_1,q_2,q_3,q_4$ are algebraically independent over $\mathbb{Z}$, then the pairing $\langle-,-\rangle_q$ is non-degenerate.
\end{cor}

\begin{proof}
For all $n$, let us consider the matrix of the pairing restricted to $\h_{\PP}(n)$ in the basis formed by the plane posets of degree $n$.
Its determinant $D_n$ is clearly an element of $\mathbb{Z}[q_1,q_2,q_3,q_4]$. Moreover, $D_n(1,0,q_3,q_4)\neq 0$ by corollary \ref{36},
so $D_n$ is a non-zero polynomial. As a consequence, if $q_1,q_2,q_3,q_4$  are algebraically independent over $\mathbb{Z}$,
$D_n(q_1,q_2,q_3,q_4)\neq 0$. \end{proof}

\section{Morphism to free quasi-symmetric functions}

\subsection{A second Hopf pairing on $\h_{(q_1,0,q,_1,q_4)}$}

We here assume that $q_2=0$ and $q_1=q_3$. For all double poset $P$, 
$$\Delta_{(q_1,0,q_1,q_4)}(P)= \sum_{\mbox{\scriptsize $I$ $h$-ideal of $P$}} 
q_1^{h_{R\setminus I}^I+r_{R\setminus I}^I}q_4^{r_I^{P\setminus I}}(P\setminus I)\otimes I.$$
For any $h$-ideal $I$ of $P$, $h_I^{P\setminus I}=0$, so:
\begin{eqnarray*}
h_{R\setminus I}^I+r_{R\setminus I}^I&=&\sharp\{(x,y)\in (P\setminus I)\times I\mid x <y\},\\
r_I^{P\setminus I}&=&r_I^{P\setminus I}+h_I^{P\setminus I}\\
&=&\sharp\{(x,y)\in (P\setminus I)\times I\mid x>y\}.
\end{eqnarray*}
Hence:
$$\Delta_{(q_1,0,q_1,q_4)}(P)= \sum_{\mbox{\scriptsize $I$ $h$-ideal of $P$}} 
q_1^{\sharp\{(x,y)\in (P\setminus I)\times I\mid x<y\}}q_4^{\sharp\{(x,y)\in (P\setminus I)\times I\mid x<y\}}(P\setminus I)\otimes I.$$

{\bf Notations.} Let $P,Q \in \P(n)$ and let $\sigma \in Bij(P,Q)$. As totally ordered sets, $P=Q=\{1,\ldots,n\}$, so $\sigma$ can be seen 
as an element of the symmetric group $\S_n$. Its \emph{length} $\ell(\sigma)$ is then its length in the Coxeter group $\S_n$ \cite{Humphreys}.

\begin{theo}We define a pairing $\langle-,-\rangle'_q:\h_q\otimes \h_q \longrightarrow K$ by:
$$\langle P,Q \rangle'_q=\sum_{\sigma \in S'(P,Q)} q_1^{\frac{n(n-1)}{2}-\ell(\sigma)}q_4^{\ell(\sigma)},$$
for all $P,Q \in \PP$, with $n=deg(P)$. Then $\langle-,-\rangle'_q$ is an homogeneous symmetric Hopf pairing on the braided Hopf algebra $\h_q=(\h,m,\Delta_q)$.
\end{theo}

\begin{proof} This pairing is clearly homogeneous. For any double posets $P$ and $Q$, the map from $S'(P,Q)$ to $S'(Q,P)$ sending $\sigma$
to its inverse is a bijection and conserves the length: this implies that the $\langle P,Q\rangle'_q=\langle Q,P\rangle'_q$. \\

Let $P,Q,R$ be three double posets. There is a bijection:
$$\Theta:\left\{\begin{array}{rcl}
Bij(PQ,R)&\longrightarrow&\displaystyle \bigcup_{I\subseteq R} Bij(P,R\setminus I)\times Bij(Q,I)\\
\sigma&\longrightarrow&(\sigma_{\mid P},\sigma_{\mid Q}),
\end{array}\right.$$
with $I=\sigma(Q)$. Let us first prove that $\sigma \in S'(PQ,R)$ if, and only if, $I$ is a $h$-ideal of $R$ and $(\sigma_1,\sigma_2) \in S'(P,R\setminus I) \times
S'(Q,I)$.

$\Longrightarrow$. Let $x'\in I$ and let $y' \in R$ such that $x'\leq_h y'$. We put $x'=\sigma(x)$ and $y'=\sigma(y)$. As $\sigma \in S'(PQ,R)$,
$x \leq y$. As $x \in Q$, $y\in Q$, so $y'\in I$: $I$ is a $h$-ideal of $R$. By restriction, $\sigma_{\mid P}\in S'(P,R\setminus I)$ and $\sigma_{\mid Q}\in S'(Q,I)$.

$\Longleftarrow$. Let us assume that $x\leq_h y$ in $PQ$. Then $x,y \in P$ or $x,y \in Q$. As  $\sigma_{\mid P}\in S'(P,R\setminus I)$
and $\sigma_{\mid Q}\in S'(Q,I)$, $\sigma(x) \leq \sigma(Q)$ in $R$.
Let us assume that $\sigma(x) \leq_h \sigma(y)$ in $R$. As $I$ is a $h$-ideal, there are two possibilities:
\begin{itemize}
\item $\sigma(x),\sigma(y) \in R\setminus I$ or $\sigma(x),\sigma(y) \in I$. Hence, $x,y \in P$ or $x,y \in Q$. 
As  $\sigma_{\mid P}\in S'(P,R\setminus I)$ and $\sigma_{\mid Q}\in S'(Q,I)$, $x \leq y$ in $PQ$.
\item $\sigma(x) \in R\setminus I$ and $\sigma(y) \in I$. Hence, $x \in P$ and $y\in Q$, so $x\leq y$ in $PQ$.
\end{itemize}

Finally, we obtain a bijection:
$$\Theta:\left\{\begin{array}{rcl}
S'(PQ,R)&\longrightarrow&\displaystyle \bigcup_{\mbox{\scriptsize $I$ $h$-ideal of $R$}} S'(P,R\setminus I)\times S'(Q,I)\\
\sigma&\longrightarrow&(\sigma_{\mid P},\sigma_{\mid Q}),
\end{array}\right.$$
Let $\sigma \in S'(PQ,R)$ and we put $\Theta(\sigma)=(\sigma_1,\sigma_2)$. Let $R\setminus I=\{i_1,\ldots,i_k\}$ and $I=\{j_1,\ldots,j_l\}$. 
Let $\zeta \in \S_{k+l}$ defined by $\zeta^{-1}(1)=i_1\ldots,\zeta^{-1}(k)=i_k,\zeta^{-1}(k+1)=j_1,\ldots,\zeta^{-1}(k+l)=j_l$.
Then $\zeta$ is a $(k,l)$-shuffle and $\sigma=\zeta^{-1} \circ (\sigma_1\otimes \sigma_2)$,
so $\ell(\sigma)=\ell(\zeta)+\ell(\sigma_1)+\ell(\sigma_2)$. Moreover:
$$\ell(\zeta)=\sharp\{(p,q)\mid j_q < i_p\}=\sharp\{(i,j)\in (R \setminus I) \times I\mid j<i\}
=r_I^{R\setminus I}+h_I^{R\setminus I}=r_I^{R\setminus I},$$
as $I$ is a $h$-ideal of $R$. Finally, $\ell(\sigma)=\ell(\sigma_1)+\ell(\sigma_2)+r_I^{R\setminus I}$. Moreover, if $|R|=n$ and $|I|=k$, by lemma \ref{13}:
$$h_{R\setminus I}^I+r_{R\setminus I}^I+r_I^{R\setminus I}=
h_{R\setminus I}^I+r_{R\setminus I}^I+h_I^{R\setminus I}+r_I^{R\setminus I}=k(n-k),$$
so:
\begin{eqnarray*}
&&h_{R\setminus I}^I+r_{R\setminus I}^I+\frac{k(k-1)}{2}+\frac{(n-k)(n-k-1)}{2}\\
&=&k(n-k)+\frac{k(k-1)}{2}+\frac{(n-k)(n-k-1)}{2}-r_I^{R\setminus I}\\
&=&\frac{n(n-1)}{2}-r_I^{R\setminus I}.
\end{eqnarray*}
Finally:
\begin{eqnarray*}
\langle P,Q \rangle'_q&=&\sum_{\sigma \in S'(PQ,R)}q_1^{\frac{n(n-1)}{2}-\ell(\sigma)}q_4^{\ell(\sigma)}\\
&=&\sum_{\mbox{\scriptsize $I$ $h$-ideal of $R$}}\sum_{\sigma_1 \in S'(P,R\setminus I),\sigma_2 \in S'(Q,I)}
q_1^{\frac{n(n-1)}{2}-\ell(\sigma_1)-\ell(\sigma_2)-r_I^{R\setminus I}}q_4^{\ell(\sigma_1)+\ell(\sigma_2)+r_I^{R\setminus I}}\\
&=&\sum_{\mbox{\scriptsize $I$ $h$-ideal of $R$}}\sum_{\sigma_1 \in S'(P,R\setminus I),\sigma_2 \in S'(Q,I)}
q_1^{h_{R\setminus I}^I+r_{R\setminus I}^I+\frac{k(k-1)}{2}+\frac{(n-k)(n-k-1)}{2}}
q_4^{\ell(\sigma_1)+\ell(\sigma_2)+r_I^{R\setminus I}}\\
&=&\sum_{\mbox{\scriptsize $I$ $h$-ideal of $R$}}q_1^{h_{R\setminus I}^I+r_{R\setminus I}^I}q_4^{r_I^{R\setminus I}}
\langle P,R\setminus I\rangle'_q\langle Q,I\rangle'_q\\
&=&\langle P\otimes Q,\Delta_q(R)\rangle'_q.
\end{eqnarray*}
So this pairing is a Hopf pairing on $\h_q$. \end{proof}\\

{\bf Remark.} In particular, $\langle-,-\rangle_{(1,0,1,1)}'$ is the pairing $\langle-,-\rangle'$ of the section \ref{recalls}.\\

{\bf Examples.} Here are the matrices of the pairing $\langle-,-\rangle'_q$ restricted to $\h_q(n)$, for $n=1,2,3$.

$$\begin{array}{c|c}
&\tun\\
\hline \tun&1
\end{array} \hspace{1cm}
\begin{array}{c|c|c}
&\tdeux&\tun\tun\\
\hline \tdeux&q_1&q_1\\
\hline \tun\tun&q_1&q_1+q_4
\end{array}$$
$$\begin{array}{c|c|c|c|c|c|c}
&\ttroisdeux&\ttroisun&\ptroisun&\tdeux\tun&\tun\tdeux&\tun\tun\tun\\
\hline \ttroisdeux&q_1^3&q_1^3&q_1^3&q_1^3&q_1^3&q_1^3\\
\hline \ttroisun&q_1^3&q_1^2(q_1+q_4)&q_1^3&q_1^2(q_1+q_4)&q_1^3&q_1^2(q_1+q_4)\\
\hline \ptroisun&q_1^3&q_1^3&q_1^2(q_1+q_4)&q_1^3&q_1^2(q_1+q_4)&q_1^2(q_1+q_4)\\
\hline \tdeux\tun&q_1^3&q_1^2(q_1+q_4)&q_1^3&q_1^2(q_1+q_4)&q_1(q_1^2+q_1^4)&q_1(q_1^2+q_1q_4+q_4^2)\\
\hline \tun\tdeux&q_1^3&q_1^3&q_1^2(q_1+q_4)&q_1(q_1^2+q_1^4)&q_1^2(q_1+q_4)&q_1(q_1^2+q_1q_4+q_4^2)\\
\hline \tun\tun\tun&q_1^3&q_1^2(q_1+q_4)&q_1^2(q_1+q_4)&q_1(q_1^2+q_1q_4+q_4^2)&q_1(q_1^2+q_1q_4+q_4^2)&
(q_1+q_4)(q_1^2+q_1q_4+q_4^2)
\end{array}$$

\begin{prop}
For all $x,y,z \in \h_q$, $\langle x\prodh y,z\rangle'_q=\langle x \otimes y, \Delta_{(q_1,0,q_1,0)}(z)\rangle'_q$.
\end{prop}

\begin{proof} Let $P,Q,R \in \PP$. We consider the following map:
$$\Theta:\left\{\begin{array}{rcl}
Bij(P\prodh Q,R)&\longrightarrow&\displaystyle \bigcup_{I\subseteq R} Bij(P,R\setminus I)\times Bij(Q,I)\\
\sigma&\longrightarrow&(\sigma_{\mid P},\sigma_{\mid Q}),
\end{array}\right.$$
with $I=\sigma(Q)$. Let us first prove that $\sigma \in S'(PQ,R)$ if, and only if, $I$ is a biideal of $R$ and $(\sigma_1,\sigma_2) \in S'(P,R\setminus I) \times
S'(Q,I)$.\\

$\Longrightarrow$. Obviously, $(\sigma_1,\sigma_2) \in S'(P,R\setminus I) \times S'(Q,I)$. Let $x' \in I$, $y'\in R$, such that $x'\leq y'$.
We put $x'=\sigma(x)$ and $y'=\sigma(y)$. If $y\notin Q$, then $y\in P$ and $x\in Q$, so $y<_h x$. As $\sigma \in S'(P\prodh Q,R)$,
$y'<x'$: contradiction. So $y\in Q$ and $y'\in I$. \\

$\Longleftarrow$. Let $x,y\in P\prodh Q$, such that $x \leq_h y$. Two cases are possible.
\begin{itemize}
\item If $x,y\in P$ or $x,y \in Q$, as $(\sigma_1,\sigma_2) \in S'(P,R\setminus I) \times S'(Q,I)$, $\sigma(x) \leq \sigma(y)$ in $R$.
\item If $x \in P$ and $y\in Q$, as $I=\sigma(Q)$ is a biideal, it is not possible to have $y\leq x$, so $x \leq y$.
\end{itemize}

Let $x,y \in P\prodh Q$, such that $\sigma(x) \leq_h \sigma(y)$ in $R$. As $I=\sigma(Q)$ is a biideal, two cases are possible.
\begin{itemize}
\item If $x,y\in P$ or $x,y \in Q$, as $(\sigma_1,\sigma_2) \in S'(P,R\setminus I) \times S'(Q,I)$, $x \leq y$ in $P\prodh Q$.
\item If $x \in P$ and $y\in Q$, then $x \leq y$ in $P \prodh Q$. \\
\end{itemize}

Moreover, if $I=\sigma(Q)$ is a biideal of $R$, then $I=\{k+1,\ldots,k+l\} \subseteq R$, where $k=|P|$ and $l=|Q|$. Then $\sigma=\sigma_1\otimes \sigma_2$,
so $\ell(\sigma)=\ell(\sigma_1)+\ell(\sigma_2)$. So, with $n=|R|$:
\begin{eqnarray*}
\langle P\prodh Q,R \rangle'_q&=&\sum_{\mbox{\scriptsize $I$ biideal of $R$}}\sum_{\sigma_1 \in S'(P,R\setminus I),\sigma_2 \in S'(Q,I)}
q_1^{\frac{n(n-1)}{2}-\ell(\sigma_1)-\ell(\sigma_2)}q_4^{\ell(\sigma_1)+\ell(\sigma_2)}\\
&=&\sum_{\mbox{\scriptsize $I$ biideal of $R$}}q_1^{\frac{n(n-1)}{2}-\frac{k(k-1)}{2}-\frac{(n-k)(n-k-1)}{2}}
\langle P\otimes Q,(R\setminus I)\otimes I\rangle'_q\\
&=&\sum_{\mbox{\scriptsize $I$ biideal of $R$}}q_1^{(n-k)k}
\langle P\otimes Q,(R\setminus I)\otimes I\rangle'_q\\
&=&\sum_{\mbox{\scriptsize $I$ biideal of $R$}}q_1^{|R \setminus I|.|I|}\langle P\otimes Q,(R\setminus I)\otimes I\rangle'_q\\
&=&\sum_{\mbox{\scriptsize $I$ biideal of $R$}}q_1^{h_{R\setminus I}^I+r_{R\setminus I}^I}\langle P\otimes Q,(R\setminus I)\otimes I\rangle'_q\\
&=&\langle P \otimes Q,\Delta_{(q_1,0,q_1,0)}(R)\rangle'_q,
\end{eqnarray*}
with the observation that, as $I$ is a biideal, $h_I^{R\setminus I}=r_I^{R\setminus I}=0$,
so $|R\setminus I|.|I|=h_{R\setminus I}^I+r_{R\setminus I}^I$. \end{proof}\\

{\bf Remark.} In particular, if $q=(1,0,1,1)$, for all $P,Q,R \in \PP$:
$$\langle P\prodh Q,R\rangle'=\langle P \otimes Q,\Delta_{(1,0,1,0)}(R)\rangle'
=\sum_{\mbox{\scriptsize $I$ biideal of $R$}} \langle P\otimes Q,(R\setminus I)\otimes I \rangle'.$$
This formula was already proved in \cite{F2}.




\subsection{Quantization of the Hopf algebra of free quasi-symmetric functions}

\begin{theo}\label{40} \begin{enumerate}
\item We define a coproduct on $\FQSym$ in the following way: for all $\sigma \in \S_n$, 
$$\Delta_q(\sigma)=\sum_{k=0}^n 
q_1^{k(n-k)-\ell(\sigma)+\ell(\sigma^{(1)}_k)+\ell(\sigma^{(2)}_k)}
q_4^{\ell(\sigma)-\ell(\sigma^{(1)}_k)-\ell(\sigma^{(2)}_k)}\sigma^{(1)}_k \otimes \sigma^{(2)}_k.$$
For all $q \in K$, $\FQSym_q=(\FQSym, m, \Delta_q)$ is a graded braided Hopf algebra.
\item One defines a Hopf pairing on $\FQSym_q$ by $\langle \sigma,\tau\rangle'=q_1^{\frac{n(n-1)}{2}-\ell(\sigma)}q_4^{\ell(\sigma)}\delta_{\sigma,\tau^{-1}}$.
\item The following map is an isomorphism of braided Hopf algebras:
$$\Theta:\left\{\begin{array}{rcl}
\h_q&\longrightarrow&\FQSym_q\\
P\in \PP&\longrightarrow&\displaystyle \sum_{\sigma \in S_P}\sigma.
\end{array}\right.$$
Moreover, for all $x,y \in \h_q$, $\langle \Theta(x),\Theta(y) \rangle'_q=\langle x,y\rangle'_q$.
\end{enumerate}\end{theo}

{\bf Remarks.} \begin{enumerate}
\item Let $\sigma \in \S_n$. If $0\leq k\leq n$, we put $\sigma(\{1,\ldots,k\})=\{i_1,\ldots,i_k\}$ and $\sigma(\{k+1,\ldots,n\})=\{j_1,\ldots,j_l\}$.
Let $\zeta$ be the $(k,l)$-shuffle defined by $\zeta^{-1}(1)=i_1\ldots,\zeta^{-1}(k)=i_k,\zeta^{-1}(k+1)=j_1,\ldots,\zeta^{-1}(k+l)=j_l$.
Then $\sigma=\zeta^{-1} \circ (\sigma^{(1)}_k \otimes \sigma^{(2)}_k)$, and $\ell(\sigma)=\ell(\zeta)+\ell(\sigma^{(1)}_k)+\ell(\sigma^{(2)}_k)$.
Moreover, $0\leq \ell(\zeta)\leq k(n-k)$, so:
$$0 \leq \ell(\sigma)-\ell(\sigma^{(1)}_k)-\ell(\sigma^{(2)}_k)\leq k(n-k).$$
As a consequence, the coproduct on $\FQSym_q$ is well-defined, even if $q_1=0$ or $q_4=0$. 
\item It is clear that the pairing $\langle-,-\rangle'_q$ defined on $\FQSym_q$ is non degenerate if, and only if, $q_1\neq 0$ and $q_2 \neq 0$.
\end{enumerate}

\begin{proof} We already know that $\Theta$ is an algebra isomorphism. Moreover, it is also proved there 
that $S'(P,Q)=S(P)\cap S(Q)^{-1}$ for any $P,Q \in PP$, so:
\begin{eqnarray*}
\langle \Theta(P),\Theta(Q)\rangle'_q&=&\sum_{\sigma \in S(P),\tau \in S(Q)} \langle \sigma,\tau\rangle'_q\\
&=&\sum_{\sigma \in S(P),\tau \in S(Q)}q_1^{\frac{n(n-1)}{2}-\ell(\sigma)}q_4^{\ell(\sigma)} \delta_{\sigma,\tau^{-1}}\\
&=&\sum_{\sigma \in S(P) \cap S(Q)^{-1}} q_1^{\frac{n(n-1)}{2}-\ell(\sigma)}q_4^{\ell(\sigma)}\\
&=&\sum_{\sigma \in S'(P,Q)} q_1^{\frac{n(n-1)}{2}-\ell(\sigma)}q_4^{\ell(\sigma)}\\
&=&\langle P,Q \rangle'_q.
\end{eqnarray*}
So $\Theta$ is an isometry. \\

As a consequence, for all $x\in \h_q$, $y,z\in \FQSym_q$:
\begin{eqnarray*}
\langle \Delta \circ \Theta(x), y\otimes z \rangle'_q&=&\langle \Theta(x),yz\rangle'_q\\
&=&\langle x, \Theta^{-1}(yz)\rangle'_q\\
&=&\langle x,\Theta^{-1}(y) \Theta^{-1}(z)\rangle'_q\\
&=&\langle \Delta_q(x),\Theta^{-1}(y) \otimes \Theta^{-1}(z)\rangle'_q\\
&=&\langle (\Theta \otimes \Theta) \circ \Delta_q(x), y\otimes z\rangle'_q.
\end{eqnarray*}
Let us now assume that $q_1,q_4\neq 0$. Then the pairing $\langle-,-\rangle'_q$ on $\FQSym_q$ is non-degenerate, so $\Delta \circ \Theta=(\Theta \otimes \Theta)\circ \Delta$.
Hence, $\Theta$ is an isometric isomorphism of braided Hopf algebras. This proves the three points immediately.

Up to an extension, we can assume that $K$ is infinite. It is not difficult to show that the set $A$ of elements $(q_1,q_4)\in K^2$ such that 
the three points are satisfied is given by polynomial equations with coefficients in $\mathbb{Z}$, As it contains $(K-\{0\})^2$ 
from the preceding observations, which is dense in $K^2$, it is is equal to $K^2$, so the result also holds for any $(q_1,q_4)$. \end{proof}

\begin{cor}
The pairing $\langle-,-\rangle'_q$ is non-degenerate if, and only if, $q_1\neq 0$ and $q_4 \neq 0$.
\end{cor}

\begin{proof} As the isometric pairing $\langle-,-\rangle'_q$ on $\FQSym_q$ is non-degenerate if, and only if, $q_1\neq 0$ and $q_4 \neq 0$. \end{proof}

\bibliographystyle{amsplain}
\bibliography{biblio}

\providecommand{\bysame}{\leavevmode\hbox to3em{\hrulefill}\thinspace}
\providecommand{\MR}{\relax\ifhmode\unskip\space\fi MR }
\providecommand{\MRhref}[2]{%
  \href{http://www.ams.org/mathscinet-getitem?mr=#1}{#2}
}
\providecommand{\href}[2]{#2}
\begin{thebibliography}{1}

\bibitem{DHT}
G{\'e}rard Duchamp, Florent Hivert, and Jean-Yves Thibon, \emph{Noncommutative
  symmetric functions. {VI}. {F}ree quasi-symmetric functions and related
  algebras}, Internat. J. Algebra Comput. \textbf{12} (2002), no.~5, 671--717,
  arXiv:math/0105065.

\bibitem{F1}
Lo{\"\i}c Foissy, \emph{Algebraic structures on double and plane posets}, to be
  published in Journal of Algebraic Combinatorics, arXiv:1101.5231, 2011.

\bibitem{F2}
\bysame, \emph{Plane posets, special posets, and permutations},
  arXiv:1109.1101, 2011.

\bibitem{FU}
Lo{\"\i}c Foissy and Jeremie Unterberger, \emph{Ordered forests, permutations
  and iterated integrals}, Int. Math. Res. Notices (2012), doi:
  10.1093/imrn/rnr273 F, arXiv:1004.5208.

\bibitem{Humphreys}
James~E. Humphreys, \emph{Reflection groups and {C}oxeter groups}, Cambridge
  Studies in Advanced Mathematics, vol.~29, Cambridge University Press,
  Cambridge, 1990.

\bibitem{Loday}
Jean-Louis Loday and Mar{\'{\i}}a Ronco, \emph{On the structure of cofree
  {H}opf algebras}, J. Reine Angew. Math. \textbf{592} (2006), 123--155,
  arXiv:math/0405330.

\bibitem{MR1}
Clauda Malvenuto and Christophe Reutenauer, \emph{Duality between
  quasi-symmetric functions and the {S}olomon descent algebra}, J. Algebra
  \textbf{177} (1995), no.~3, 967--982.

\bibitem{MR2}
Claudia Malvenuto and Christophe Reutenauer, \emph{A self-dual {H}opf algebra
  on double partially ordered sets}, arXiv:0905.3508.

\end{thebibliography}

\end{document}